\documentclass[a4paper,11pt]{article}
\usepackage{graphicx,a4wide}
\usepackage[intlimits]{amsmath}
\usepackage{amsxtra, amssymb, amsthm,latexsym}
\usepackage[english]{babel}
\usepackage{enumerate}
\numberwithin{equation}{section}
\numberwithin{figure}{section}
\usepackage{color}
\usepackage{caption}
\usepackage{subcaption}

\begin{document}

\newtheorem{theorem}{Theorem}[section]
\newtheorem{lemma}[theorem]{Lemma}
\newtheorem{claim}[theorem]{Claim}
\newtheorem{proposition}[theorem]{Proposition}
\newtheorem{postulate}[theorem]{Postulate}
\newtheorem{definition}[theorem]{Definition}
\newtheorem{assumption}[theorem]{Assumption}
\newtheorem{corollary}[theorem]{Corollary}
\theoremstyle{definition}
\newtheorem{remark}[theorem]{Remark}

% The following commands have been added by Sander Hille in order to avoid problems with the package bbm:
\newcommand{\mathbbm}[1]{{#1\!\!#1}}
\newcommand{\ind}{{\mathbbm{1}}}

% Further new commands:
\newcommand{\CM}{\mathcal{M}}
\newcommand{\CMc}{\overline{\CM}}
\newcommand{\smfrac}[2]{\mbox{$\frac{#1}{#2}$}}
\newcommand{\geqs}{\geqslant}
\newcommand{\leqs}{\leqslant}
\newcommand{\nn}{{(n)}}
\newcommand{\mm}{{(m)}}

% === end of commands added ===

\newcommand{\supp}{\operatorname*{supp}}
\newcommand{\BL}{{\mathrm{BL}}}
\newcommand{\TV}{{\mathrm{TV}}}
\newcommand{\FM}{{\mathrm{FM}}}
\newcommand{\pair}[2]{\left\langle #1 , #2 \right\rangle}
\newcommand{\Int}[4]{\int_{#1}^{#2}\! #3 \, #4}
\def\R{\mathbb{R}}
\def\Rp{\mathbb{R}^+}
\def\N{\mathbb{N}}
\def\Np{\mathbb{N}^+}
\def\eps{\varepsilon}
\def\aeps{a}

\title{Mild solutions to a measure-valued mass evolution problem with flux boundary conditions}

% \title{Deriving suitable boundary conditions for a 1D measure-theoretical mass evolution problem}

\author{Joep H.M.~Evers\thanks{CASA - Centre for Analysis, Scientific computing and Applications, ICMS - Institute for Complex Molecular Systems, Eindhoven University of Technology, The Netherlands. This work is financially supported by the Netherlands Organisation for Scientific Research (NWO), Graduate Programme 2010.} \  \  Sander C.~Hille\thanks{Corresponding author; email: \texttt{shille@math.leidenuniv.nl}} \thanks{Mathematical Institute, Leiden University, P.O. Box 9512, 2300 RA, Leiden, The Netherlands}  \  \ Adrian Muntean\thanks{CASA - Centre for Analysis, Scientific computing and Applications, ICMS - Institute for Complex Molecular Systems, Eindhoven University of Technology, The Netherlands}}

%\author{Joep Evers, Sander Hille\thanks{Corresponding author}, Adrian Muntean}

% \subjclass[2000]{...}
% \keywords{measure-valued solution, flux boundary condition, boundary layer, singular limit, nonsmooth analysis}

\date{\today}
\maketitle

%\abstract{Motivated by our research on pedestrian flows, we study a non-conservative measure-valued  evolution problem posed in a finite interval and explore the possibility of imposing a flux boundary condition. The main steps of our work include the analysis of a suitably scaled regularized problem possessing a boundary layer that accumulates mass, and detailed investigations of the boundary layer by means of semigroup techniques in spaces of measures. We consider passage to the singular limit where thickness of the layer vanishes (resembling the fast reaction asymptotics typical for systems with slow transport and  rapid reactions). We obtain not only suitable solutions to the measure-valued evolution problem, but also derive a convergence rate for the approximation procedure as well as the structure of (flux) boundary conditions for the limit problem.}

\begin{abstract}
We investigate the well-posedness and approximation of mild solutions to a class of linear transport equations on the unit interval $[0,1]$ endowed with a linear discontinuous production term, formulated in the space $\CM([0,1])$ of finite Borel measures. Our working technique includes a detailed boundary layer analysis in terms of a semigroup representation of solutions in spaces of measures able to cope with the passage to the singular limit where thickness of the layer vanishes. We obtain not only a suitable concept of  solutions to the chosen measure-valued evolution problem, but also derive convergence rates for the approximation procedure and get insight in the structure of flux boundary conditions for the limit problem.

%Motivated by our research on pedestrian flows, we study a non-conservative measure-valued  evolution problem posed in a finite interval and explore the possibility of imposing a flux boundary condition. The main steps of our work include the analysis of a suitably scaled regularized problem possessing a boundary layer that accumulates mass, and detailed investigations of the boundary layer by means of semigroup techniques in spaces of measures. We consider passage to the singular limit where thickness of the layer vanishes (resembling the fast reaction asymptotics typical for systems with slow transport and  rapid reactions). We obtain not only suitable solutions to the measure-valued evolution problem, but also derive a convergence rate for the approximation procedure as well as the structure of (flux) boundary conditions for the limit problem.
\end{abstract}

{\small \noindent{\it Keywords:}\ Measure-valued equations,  flux boundary condition, mild solutions, boundary layer asymptotics, singular limit, nonsmooth analysis, convergence rate, particle systems

}

%\tableofcontents

\section{Introduction}

Measure-valued formulations of evolution equations have an increasing interest, both from fundamental mathematical perspective (e.g. gradient flows in metric spaces \cite{Savare}) and in applications (e.g. in structured population models \cite{Ackleh1, Ackleh2,Gwiazda2,DG,Gwiazda1}, crowd dynamics \cite{Bellomo,PiccoliTosinMeasTh}). Pedestrian crowds and their dynamics in high-density regimes is a modern topic of intense study not only in security, logistics, and civil engineering (crowd-structure interactions) but also in non-equilibrium statistical mechanics of social systems; see  e.g. \cite{HelbingMolnar} and \cite{Schadschneider2011} for an overview of the current status of research. Other related examples can be found in \cite{Borsche} and references cited therein.

Here we consider the problem of well-posedness and approximation of solutions to a class of linear transport equations on the unit interval $[0,1]$ with linear, but discontinuous perturbation, formulated in the space $\CM([0,1])$ of finite Borel measures. That is, we prove existence, uniqueness and continuous dependence of solutions on initial data for the measure-valued equation
\begin{equation}\label{eq:main equation}
\frac{\partial}{\partial t}\mu_t + \frac{\partial}{\partial x} (v\mu_t) = f\cdot \mu_t,\qquad\mbox{on}\ [0,1],
\end{equation}
where $f:[0,1]\to\R$ is a {\em piecewise bounded-Lipschitz function} with finitely many discontinuities. That is, $f$ has finitely many points of discontinuity and the restriction of $f$ to each of the intervals of continuity is bounded Lipschitz. $f\cdot \mu_t$ is the measure on $[0,1]$ with density $f$ with respect to $\mu_t$. We assume that the velocity field $v:[0,1]\to\R$ is bounded Lipschitz and interpret a $v$ that points outwards at either one of the boundary points $x=0$ or $x=1$ (i.e. $v(0)<0$ or $v(1)>0$) as describing the presence of a `sticking boundary' (cf. \cite{Taira}) at that point. Alternatively, other researchers (e.g. \cite{Gwiazda-Jamroz_ea:2012}) have replaced $v$ by a discontinuous $\hat{v}$ that equals $v$ on $(0,1)$, but is set to zero at boundary points where $v$ points outwards.

Equation \eqref{eq:main equation} models e.g. mass transport over a conveyor belt (cf. \cite{Schleper}, e.g.), with mass accumulating in a collector at the end of the belt, while mass is added to (or removed from) the belt or collector locally at $x$ at a rate proportional to the mass present, with proportionality constant $f(x)$. The discontinuities in $f$ in our setting allow modelling abrupt changes in these addition and removal processes. The equation also occurs in the analysis of recent models for biological cell maturation (cf. \cite{Gwiazda-Jamroz_ea:2012}, where a slightly different formulation is used).

Formulation \eqref{eq:main equation} unifies a continuum formulation in terms of density functions with respect to the  Lebesgue measure and a particle description for this mass evolution problem within the framework of measure-valued differential equations.  As prominent feature of this framework, well-posedness together with Lyapunov stability of measure-valued solutions (for suitable metrics) imply that particle dynamics and continuum solutions will stay close once initial conditions are sufficiently close. We investigate this resemblance numerically in Sections \ref{section: two approaches numerics} and \ref{subsection num results}.

If $f$ in \eqref{eq:main equation} were bounded Lipschitz, then the proof of well-posedness would follow standard arguments for semilinear equations (e.g. see \cite{Pazy,Lunardi,Cazenave-Haraux}), except for the technical point that $\CM([0,1])$ is not complete for the natural norms used in measure-valued equations (see also \cite{Canizo}). These are the equivalent Fortet-Mourier norm $\|\cdot\|_\FM^*$ and dual bounded Lipschitz norm $\|\cdot\|_\BL^*$ (also known as Dudley norm or `flat metric' \cite{Dudley1,Gwiazda1}) that metrize the weak-star topology when restricted to the cone of positive measures \cite{Dudley1,Dudley2}. In our setting, with piecewise bounded Lipschitz $f$, the perturbation map
\begin{equation}
F_f:\CM([0,1])\to\CM([0,1]): \mu\mapsto f\cdot \mu
\end{equation}
is not Lipschitz, but (mildly) discontinuous. Nevertheless, continuous dependence on initial conditions still holds, remarkably, using a different approach.

The particular choice $f(x)=-a\ind_{\{1\}}(x)$, where $\ind_E$ is the indicator function of the set $E$ and $a> 0$, reduces equation \eqref{eq:main equation} to
\begin{equation}
\frac{\partial}{\partial t}\mu_t + \frac{\partial}{\partial x} (v\mu_t) = -a\mu_t(\{1\})\delta_1,
\label{eq:boundary condition}
\end{equation}
where $\delta_1$ denotes the Dirac measure at $1$.
If $v(1)>0$, then \eqref{eq:boundary condition} represents a system with mass transport and sticking boundaries \cite{Taira} of which $x=1$ is {\em partially absorbing}: mass arriving at $x=1$ stays there, while it is removed at a constant rate $a$. This represents a {\em flux boundary condition}, or Robin-like boundary condition, in a measure-valued formulation. Note that in this formulation the flux condition appears as a discontinuous perturbation term in the equation, that is concentrated on the boundary. This circumvents the introduction of some concept of normal derivative of measures on the boundary.

Searching for the correct flux boundary conditions is a topic often addressed in the literature, but this has not yet been treated in the measure-theoretical framework. We refer here for instance  to \cite{Arrieta, Schochet,Casas, Muratov} (in the context of reaction and diffusion scenarios) and Gurtin \cite{Gurtin} (the shrinking pillbox principle in continuum mechanics).

In the part on approximation by regularization in this paper (Section \ref{sec:regularized systems}) the piecewise bounded Lipschitz function $f$ in \eqref{eq:main equation} is replaced by a bounded Lipschitz function $f_\eps$ that differs from $f$ on an $\eps$-thin neighbourhood of each discontinuity. We show, that if we pick a sequence of such functions $f_n=f_{\eps_n}$ with $\eps_n\downarrow 0$, then the corresponding solutions $\mu^\nn_t$ converge to the solution $\mu_t$ of \eqref{eq:main equation}. Moreover, we compute the rate of convergence. This result shows on the one hand that \eqref{eq:main equation} with instantaneous spatial change in $f$ can be viewed as an idealisation of a continuous (Lipschitzian), but very fast, change. On the other hand, it shows that \eqref{eq:boundary condition} indeed represents a flux boundary condition, as it results from appropriate interaction with the boundary in a thin boundary layer, in the limit of vanishing thickness.

Our approach to equation \eqref{eq:main equation} differs from those cited above that use appropriate weak solution concepts. Ours relies on the related preliminary studies \cite{EversMuntean,CRAS} (crowd dynamics, balanced mass measures) and \cite{Sander} (measures in the semigroups context). The notion of solution to \eqref{eq:main equation} introduced here is that of {\em measure-valued mild solution}. We interpret equation \eqref{eq:main equation} as expressing formally that the semigroup $(P_t)_{t\geqs0}$ of operators on $\CM([0,1])$ associated to mass transport along characteristics defined by the velocity field $v$ is perturbed by $F_f$. A mild solution is then a continuous map $t\mapsto \mu_t$ from an interval $[0,T]$ into $\CMc([0,1])_\BL$, which is the completion of $\CM([0,1])]$ equipped with $\|\cdot\|_\BL^*$, that satisfies the variation of constants formula
\begin{equation}\label{eq:VoC}
\mu_t = P_t\mu_0 + \int_0^t P_{t-s}F_f(\mu_s)\,ds,\qquad 0\leqs t\leqs T.
\end{equation}
Although $F_f$ is discontinuous, the map $s\mapsto F_f(\mu_s)$ is Bochner measurable for such map $\mu_\bullet$. Thus, integral equation \eqref{eq:VoC} is well-defined. In our solution concept we need to include the technical condition that the total variation function $t\mapsto \|\mu_t\|_\TV$ of the solution must be bounded on $[0,T]$. (See Section \ref{sec:model formulation} for further discussion).

The study of linear equations like \eqref{eq:main equation} is a first step in the study of equations with time dependent or even density dependent velocity field.  This approach connects to the approach to structured population models as found in e.g. \cite{DG, Gwiazda1,Gwiazda2,Ackleh1, Ackleh2}. See also Canizo {\it et al.} \cite{Canizo} that simplifies part of the approach in \cite{Gwiazda1,Gwiazda2}. It consists of first studying the situation of independently moving individuals described by a (semi-)linear system like \eqref{eq:main equation}. Then use the obtained results to treat the case of forced velocity fields changing in time and finally close the system for a density dependent velocity field, which models interacting individuals. This paper addresses the first step in this `program' for a system with boundary conditions.

%{\bf (** Mention non-linear settings of interst: `The type of nonlinear structure in  (\ref{structure2}) is the main point of the setting addressed by Fibich {\it et al.} in \cite{Schochet}'. **)}

%{\bf (** `In order to focus on the essence of mathematical issues that arise when having a sticky, partially absorbing boundary, we focus on a situation where there is only one such boundary and this boundary point can be reached. That is, we assume $v(0)>0$ and $v(1)>0$, such that the boundary point $1$ is the only one of interest.' **)}

\subsection{Organization of the paper}

%{\bf (** This section has to be completely revised... **)}

For facilitating the reading of the paper, Section \ref{sec:preliminaries} collects a series of properties and results on finite Borel measures, topologies on spaces of such measures and associated metrising norms and Bochner integrals involving measure-valued kernels that will be used throughout the paper. Section \ref{sec:model formulation} sets-up the details of the  mathematical model on which we are focussing our attention. Section \ref{sec:inidividualistic flow} derives fundamental estimates for the movement of single individuals in the domain, i.e. the {\em individualistic flow}, where there is no absorption yet, but sticking boundaries only. After recalling basic semigroup estimates, we bring the attention of the reader to a technical lemma concerning regularisation by averaging over orbits (see Lemma \ref{lem:crucial BL-function}) that will play a crucial role in obtaining the main results. The concept of mild solution formulated in the space $\CM([0,1])$ of finite Borel measures is introduced in Section \ref{mild}.
In Section \ref{sec:limit well-postness ad hoc}, we prove the well-posedness of our problem and we address the issues concerning the approximation of the mild solution in Section \ref{sec:regularized systems}. The numerical approximation of a study case is treated in Section \ref{section: numerics}, where we  compute numerically the theoretically predicted order of convergence for a discrete measure-valued solution. The paper closes with a section that provides a probabilistic underpinning of our solution concept. Three technical appendices containing proofs of the used properties of the individualistic stopped flow, of our averaging lemma over orbits, as well as the proof of some basic results concerning the integration of measure-valued maps complete the paper.

%In Section \ref{sec:alternative extension} we discuss an alternative approach that consists of extension of the domain to an unbounded setting that may seem a way to avoid some technicalities. We show that it suffers from the same type of technical difficulties.  Section \ref{sec:regularized systems} discusses systems in which there is interaction in a (thin) layer near the boundary. There, well-posedness and positivity for measure-valued solutions is proven for these systems. Finally, Section \ref{sec:solution limit eq} deals with the integral equation \eqref{eq:essential VOC formula}. Well-posedness is shown in Section \ref{sec:limit well-postness ad hoc} exploiting the convenient structure of the system. The most important parts are Section \ref{sec:existence regularized} and \ref{sec:rate convergence}. In the first the unique solution to \eqref{eq:essential VOC formula} is shown to arise as limit of the solutions of a sequence of thin boundary layer solutions, in the limit of vanishing thickness of the layer. Finally, Section \ref{sec:rate convergence} computes the rate of convergence of these boundary layer solutions to the solutions of the limit equation.

\subsection{Preliminaries on measures}
\label{sec:preliminaries}

If $S$ is a topological space, we denote by $\CM(S)$ the space of finite Borel measures on $S$ and $\CM^+(S)$ the convex cone of positive measures included in it.
%This cone defines a partial ordering on $\CM(S)$: $\mu\leqs \nu$ iff $\nu-\mu\in\CM^+(S)$. Clearly, $\mu\leqs \nu$ if and only if $\mu(E)\leqs \nu(E)$ for all Borel measurable $E\subset S$.
For $x\in S$, $\delta_x$ denotes the Dirac measure at $x$. Put
\begin{equation}\label{pairing}
\langle\mu,\phi\rangle :=\int_S \phi \,d\mu,
\end{equation}
the natural pairing between measures $\mu\in\CM(S)$ and bounded measurable functions $\phi$. If $\Phi:S\to S$ is Borel measurable, the {\em push forward} or {\em image measure} of $\mu$ under $\Phi$ is the measure $\Phi\#\mu$ defined on Borel sets $E\subset S$ by
\[
(\Phi\# \mu)(E) := \mu\bigl(\Phi^{-1}(E)\bigr).
\]
It is easily verified that $\langle\Phi\# \mu,\phi\rangle=\langle\mu,\phi\circ\Phi\rangle$.

We denote by $C_b(S)$ the Banach space of real-valued bounded continuous functions on $S$ equipped with the supremum norm $\|\cdot\|_\infty$. The {\em total variation norm} $\|\cdot\|_\TV$ on $\CM(S)$ is given by
\begin{equation}
\|\mu\|_{\TV}:= \sup\left\{\pair{\mu}{\phi}\,\Big|\, \phi\in C_b(S),\ \|\phi\|_\infty\leqslant 1 \right\}.
\end{equation}
It follows immediately that for $\Phi:S\to S$ continuous, $\|\Phi\#\mu\|_\TV\leqs\|\mu\|_\TV$.

The total variation norm is too strong for our application, since $\|\delta_x-\delta_y\|_\TV = 2$ if $x\neq y$. The natural topology to consider is the weak topology induced by $C_b(S)$ through the pairing \eqref{pairing}. In this topology $x\mapsto\delta_x:S\to\CM^+(S)$ is continuous.

In our setting, $S=[0,1]$ is a Polish space. It is well-established (cf. \cite{Dudley1,Dudley2}) that in this case the weak topology {\em on the positive cone} $\CM^+(S)$ is metrisable by a metric derived from a norm, e.g. the Fortet-Mourier norm or the Dudley norm, which is also called the {\em dual bounded Lipschitz norm}, that we shall introduce now. To that end, let $d$ be a metric on $S$ that metrises the topology, such that $(S,d)$ is separable and complete. Let $\BL(S,d)=\BL(S)$ be the vector space of real-valued bounded Lipschitz functions on $(S,d)$. For $\phi\in\BL(S)$, let
\begin{equation}
|\phi|_L := \sup\left\{ \frac{|\phi(x)-\phi(y)|}{d(x,y)}\;\Big|\; x,y\in S,\ x\neq y\right\}
\end{equation}
be its Lipschitz constant. Then
\begin{equation}\label{def:BL-norm}
\|\phi\|_\BL := \|\phi\|_\infty + |\phi|_L
\end{equation}
defines a norm on $\BL(S)$ for which it is a Banach space \cite{Fortet-Mourier:1953,Dudley1}. In fact it makes $\BL(S)$ a Banach algebra for pointwise product of functions:
\begin{equation}
\|\phi\cdot\psi\|_\BL \le \|\phi\|_\BL\,\|\psi\|_\BL.
\end{equation}
Alternatively, one may define on $\BL(S)$ the equivalent norm
\begin{equation}
\|f\|_\FM := \max\bigl(\|f\|_\infty,|f|_L\bigr).
\end{equation}
Let $\|\cdot\|_\BL^*$ be the dual norm of $\|\cdot\|_\BL$ on the dual space $\BL(S)^*$, i.e. for any $x^*\in\BL(S)^*$:
\begin{equation}
\|x^*\|_{\BL}^*:= \sup\left\{ |\pair{x^*}{\phi}|\ |\; \phi\in\BL(S),\ \|\phi\|_\BL\leqslant1\right\}.
\end{equation}
The map $\mu\mapsto I_\mu$ with $I_\mu(\phi):=\pair{\mu}{\phi}$ defines a linear {\em embedding} of $\CM(S)$ into $\BL(S)^*$ (\cite{Dudley1}, Lemma 6). Thus $\|\cdot\|^*_\BL$ induces a norm on $\CM(S)$, which is denoted by the same symbols. It is called the dual bounded-Lipschitz norm or Dudley norm. Generally, $\|\mu\|_\BL^*\leqslant\|\mu\|_\TV$. For positive measures the norms coincide:
\begin{equation}\label{TV norm is dual BL norm for pos measures}
\|\mu\|_{\BL}^*=\mu(S)=\|\mu\|_{\TV} \hspace{1 cm}\text{for all }\mu\in\CM^+(S).
\end{equation}
One may also consider the restriction to $\CM(S)$ of the dual norm $\|\cdot\|_\FM^*$ of $\|\cdot\|_\FM$ on $\BL(S)^*$. This yields an equivalent norm on $\CM(S)$ that is called the Fortet-Mourier norm (see e.g. \cite{Lasota-Myjak-Szarek,Zaharopol}):
\begin{equation}\label{eq:equivalence Dudely Fortet-Mourier}
\|\mu\|^*_\BL\leqs\|\mu\|^*_\FM\leqs 2\|\mu\|^*_\BL.
\end{equation}
It also satisfies $\|\mu\|^*_\FM\leqs\|\mu\|_\TV$, so \eqref{TV norm is dual BL norm for pos measures} holds for $\|\cdot\|_\FM^*$ too. Moreover (cf. \cite{Hille-Worm:2009}, Lemma 3.5), for any $x,y\in S$,
\begin{equation}\label{eq:Bl-norm difference Diracs}
\|\delta_x-\delta_y\|_\BL^* = \frac{2d(x,y)}{2+d(x,y)} \leqs \min(2,d(x,y)) = \|\delta_x-\delta_y\|_\FM^*.
\end{equation}

The $\|\cdot\|_\BL^*$-norm topology on $\CM^+(S)$ coincides with the restriction of the weak topology (\cite{Dudley1}, Theorem 12). $\CM(S)$ is not complete for $\|\cdot\|_\BL^*$ generally. We denote by $\CMc(S)_\BL$ its completion, viewed as closure of $\CM(S)$ within $\BL(S)^*$. $\CM^+(S)$ is complete for $\|\cdot\|^*_\BL$, hence closed in $\CM(S)$ and $\CMc(S)_\BL$.
Note that $\BL(S,d)$ will vary with $d$, hence $\|\cdot\|_\BL^*$ on $\CM(S)$ will depend on $d$ too and so will $\CMc(S)_\BL$.

The $\|\cdot\|_\BL^*$-norm is convenient also for integration. In Appendix \ref{sec:integration meas-valued maps} some technical results that are used in this paper have been collected for the reader's convenience. The continuity of the map $x\mapsto\delta_x:S\to\CM^+(S)_\BL$ together with \eqref{eq:set within integral} yields the identity
\begin{equation}\label{eq:superposition of Diracs}
\mu = \int_S\delta_x\,d\mu(x)
\end{equation}
as Bochner integral in $\CM(S)_\BL$. This observation will essentially link continuum (`$\mu$') and particle description (`$\delta_x$') for our linear equation on $[0,1]$.

\section{Model formulation and solution concept}
\label{sec:model formulation}

%{\bf (** To be revised: }
%An `educated guess' may lead directly to equation \eqref{structure1}. In cases where the dynamics is more complicated -- in higher dimensions, in the interior of the domain, at the boundary -- it may not be that evident. In this section we explain how \eqref{structure1} can be derived by considering a thin boundary layer in which well-described interactions with the boundary occur, and letting its thickness vanish in a limit. Such an approach would allow to determine the manner in which boundary interactions should be incorporated in the measure-valued formulation. In \cite{Schochet} (in a non-measure setting of concentration functions and surface densities) such an approach was used to establish the right boundary conditions for chemical reactions occuring on a surface in a three dimensional domain. {\bf **)}

Throughout the remainder of the paper, $f:[0,1]\to\R$ will be a piecewise bounded-Lipschitz function as defined in the Introduction. $v:[0,1]\to\R$ is a bounded Lipschitz velocity field.

\subsection{Mass transport and averaging along characteristics}
\label{sec:inidividualistic flow}

We assume that a single particle (`individual') is moving in the domain $[0,1]$ deterministically, described by the differential equation for its position $x(t)$ at time $t$:
\begin{equation}\label{eq:indiv flow}
\left\{
  \begin{array}{l}
    \dot{x}(t)=v(x(t)), \\
    x(0)=x_0.
  \end{array}
\right.
\end{equation}
Thus, a solution to \eqref{eq:indiv flow} is unique, it exists for time up to reaching the boundary $0$ or $1$ and depends continuously on initial conditions. Let $x(\cdot;x_0)$ be this solution and $I_{x_0}$ be its maximal interval of existence. Put
\[
\tau_\partial(x_0) := \sup I_{x_0} \in [0,\infty],
\]
i.e. $\tau_\partial(x_0)$ is the time at which the solution starting at $x_0$ reaches the boundary (if it happens) when $x_0$ is an interior point. Note that $\tau_\partial(x_0)=0$ when $x_0$ is a boundary point where $v$ points outwards, while $\tau_\partial(x_0)>0$ when $x_0$ is a boundary point where $v$ vanishes or points inwards.\\

\noindent The {\em individualistic stopped flow} on $[0,1]$ associated to $v$ is the family of maps $\Phi_t:[0,1]\to[0,1]$, $t\geqs 0$,  defined by
\begin{equation}\label{individualistic flow Phi}
\Phi_t(x_0) := \begin{cases} x(t;x_0),& \quad \mbox{if}\ t\in I_{x_0},\\
x(\tau_\partial(x_0);x_0), & \quad \mbox{otherwise}.
\end{cases}
\end{equation}
Below we collect important properties of the family of maps $(\Phi_t)_{t\geqs 0}$ in a series of lemmas.
%We write $|v|_L$ instead of the more precise $|\bar{v}|_L$ or $|v|_{(0,1)}|_L$.

\begin{lemma}\label{lem:basic props indiv flow}
$(\Phi_t)_{t\ge 0}$ is a semigroup of Lipschitz transformations of $[0,1]$. Moreover,
\begin{enumerate}
\item[(i)] $|\Phi_t|_L \leqs e^{|v|_L t}$ for $t\geqs 0$.
\item[(ii)] For any $t,s\in\Rp$,
\begin{equation}\label{eq:uniform Lipschitz in time}
\sup_{x\in[0,1]} |\Phi_t(x)-\Phi_s(x)| \leqs \|v\|_\infty\, |t-s|.
\end{equation}
\end{enumerate}
\end{lemma}
\begin{proof} See Appendix \ref{app:proofs lemmas} for the proof of this lemma.
\end{proof}

Define $P_t$ to be the lift of $\Phi_t$ to the space of finite Borel measures $\CM([0,1])$ by means of push forward under $\Phi_t$: for all $\mu\in\mathcal{M}([0,1])$,
\begin{equation}\label{Def Pt push forward}
P_t\mu := \Phi_t \# \mu = \mu\circ\Phi_t^{-1}.
\end{equation}
Clearly, $P_t$ maps positive measures to positive measures and $P_t$ is mass preserving on positive measures. Since the maps $\Phi_t$, $t\geqs 0$, form a semigroup, so do the maps $P_t$ in the space $\CM([0,1])$. That is, $(P_t)_{t\geqs0}$ is a {\em Markov semigroup} on $[0,1]$ (cf. \cite{Lasota-Myjak-Szarek}). One has $\|P_t\mu\|_\TV \leqs \|\mu\|_\TV$ for general $\mu\in\CM([0,1])$.

\begin{lemma}\label{lemma:lift is Lipschitz in time}
Let $\mu\in\CM([0,1])$ and $t,s\in\Rp$. Then
\begin{enumerate}
\item[({\it i})] $\| P_t\mu- P_s\mu\|_\BL^* \leqs \|v\|_\infty\,\|\mu\|_\TV \,|t-s|$.

\item[({\it ii})] $\| P_t\mu\|_\BL^* \leqs \max(1,|\Phi_t|_L)\, \|\mu\|_\BL^* \leqs e^{|v|_Lt}\|\mu\|_\BL^*$.
\end{enumerate}
\end{lemma}

\begin{proof}
For all $\phi\in\BL([0,1])$,
\begin{align*}
|\pair{P_t\mu - P_s \mu}{\phi}| &=  |\pair{\mu}{\phi\circ\Phi_t - \phi\circ\Phi_s}|
\leqslant \|\mu\|_{\TV} \, \|\phi\circ\Phi_t - \phi\circ\Phi_s \|_\infty \\
&\leqslant \|\mu\|_{\TV} \, |\phi|_L \, \underset{x\in[0,1]}{\sup} |{\Phi}_t(x) - {\Phi}_s(x)| \leqslant \|\mu\|_{\TV} \, |\phi|_L \, \|v\|_\infty \, |t-s|,
\end{align*}
where we used Lemma \ref{lem:basic props indiv flow} ({\it ii}) in the last inequality. For the operator norm of $P_t$ on $\CM([0,1])_\BL$, use that for any $\phi\in\BL([0,1])$,
\[
|\langle P_t\mu,\phi\rangle| = |\langle\mu,\phi\circ \Phi_t\rangle| \leqs \|\mu\|_\BL^*\, \|\phi\circ\Phi_t\|_\BL \leqs \|\mu\|_\BL^*\,\bigl( \|\phi\|_\infty + |\phi|_L|\Phi_t|_L\bigr).
\]
The statement in the lemma follows.
\end{proof}

\noindent For any bounded measurable function $g$ on $[0,1]$ and $t\geqs0$, define the average over partial orbits under the flow $\Phi_t$ as the function $g^\Phi_t:[0,1]\to\R$ given by
\begin{equation}\label{def:integral flow}
g^\Phi_t(x):= \int_0^t g(\Phi_s(x))\,ds.
\end{equation}
Clearly,
\begin{equation}\label{eq:Lipschitz prop If}
| g^\Phi_t(x)-g^\Phi_{t'}(x)| \leqs \|g\|_\infty\cdot |t-t'|,
\end{equation}
so $t\mapsto g^\Phi_t(x)$ is Lipschitz for any $x\in[0,1]$.

The following lemma is crucial for establishing the continuous dependence on initial conditions. % Its tedious proof is provided in Appendix \ref{app:proofs lemmas}.

\begin{lemma}[Regularisation by averaging over orbits]\label{lem:crucial BL-function}
Let $(\Phi_t)_{t\geqs 0}$ be the individualistic stopped flow associated to the velocity field $v$. Let $g$ be a piecewise bounded Lipschitz function on $[0,1]$ such that $v(x)\neq0$ at any point of discontinuity of $g$. Then $g^\Phi_t$ is a bounded Lipschitz function on $[0,1]$ for any $t\geqs 0$. Moreover,
\begin{equation}\label{eq:upper bound Lipschitz int f}
\sup_{0\leqs s\leqs t} |g^\Phi_s|_L < \infty.
\end{equation}
\end{lemma}
\begin{proof} See Appendix \ref{regularisation-orbits} for the proof of this lemma.
\end{proof}

Recall that we assume that $v$ is bounded Lipschitz, hence continuous, on the {\em entire} interval $[0,1]$, because we do not replace $v$ by a function that is zero at boundary points where $v$ points outwards. So, the conditions of Lemma \ref{lem:crucial BL-function} allow for the situation that $g$ has a discontinuity at the boundary. This is a case of particular interest in view of a `flux' boundary condition in a measure-valued formulation. (See Section \ref{sec:regularized systems}, Example, for further discussion of this point).

\subsection{Mild solutions}\label{mild}

Measure-valued equations on $\Rp$ or $\R^d$ of the form \eqref{eq:main equation}, among others, have been studied in e.g. \cite{Canizo,Gwiazda1}, where a concept of {\em weak solution} to the associated Cauchy problem was introduced and subsequently existence of such weak solution and continuous dependence on initial data was proven. \cite{Gwiazda-Jamroz_ea:2012} considers (in Section 3.1) a situation that is similar to ours by considering a measure-valued transport equation on a finite interval $[0,x^*]$ as \eqref{eq:main equation} with $f=0$ in which the velocity field $v$ is piecewise bounded-Lipschitz, continuous at $0$ with $v(0)>0$ and $v(x)=0$ at any point of discontinuity, including the other boundary point $x^*$. There, again the weak solution concept is employed (cf. \cite{Gwiazda-Jamroz_ea:2012}, Definition 3.4, and references provided).

In this paper, we introduce and use the concept of {\em measure-valued mild solution} to the Cauchy problem associated to equation \eqref{eq:main equation} in the spirit of the semigroup approach to semilinear evolution equations (e.g. \cite{Pazy,Lunardi,Cazenave-Haraux}). That is, we interpret the operator $\mu\mapsto -\frac{\partial}{\partial x}(v\mu)$ as generator of a strongly continuous semigroup in $\CMc([0,1])_\BL$: the semigroup $(P_t)_{t\geqs 0}$ of mass transport along characteristics associated to the velocity field $v$ that we defined in Section \ref{sec:inidividualistic flow}. Equation \eqref{eq:main equation} is then viewed as perturbation of this semigroup by means of $F_f:\mu\mapsto f\cdot\mu$, which is defined on the dense subspace $\CM([0,1])$ of $\CMc([0,1])_\BL$ only. Generally, $F_f$ is not $\|\cdot\|_\BL^*$-continuous, unless $f\in\BL([0,1])$. It is unclear whether $F_f$ is a densely-defined closed linear map. So `standard' perturbation results cannot be readily applied.

Moreover, we have to adapt slightly the classical concept of mild solution, since we are interested in measure-valued solutions while the Banach space $\CMc([0,1])_\BL$ also contains points that are not measures, as closure of $\CM([0,1])$ in $\BL([0,1])^*$.
\begin{definition}
A {\em measure-valued mild solution} to the Cauchy-problem associated to \eqref{eq:main equation} on $[0,T]$ with initial value $\nu\in\CM([0,1])$ is a continuous map $\mu_\bullet:[0,T]\to\CM([0,1])_\BL$ that is $\|\cdot\|_\TV$-bounded and that satisfies the variation of constants formula
\begin{equation}\label{eq:VOC}
\mu_t = P_t\,\nu \ +\ \int_0^t P_{t-s}F_f(\mu_s)\, ds\qquad \mbox{for all}\ t\in [0,T].
\end{equation}
\end{definition}

\noindent{\bf Remarks.}\ 1.) The map $s\mapsto F_f(\mu_s)$ will not be continuous. It is Bochner measureable though: there exist $f_n\in\BL([0,1])$ such that $f_n\to f$ pointwise while $\sup_n\|f_n\|_\infty<\infty$. Each map $s\mapsto F_{f_n}(\mu_s)$ is continuous, hence Bochner measureable, and $F_f(\mu_\bullet)$ is the pointwise limit of $F_{f_n}(\mu_\bullet)$. Consequently, $s\mapsto P_{t-s}F_f(\mu_s)$ is Bochner measurable and the integral in \eqref{eq:VOC} is well-defined as Bochner integral in $\CMc([0,1])_\BL$.\\
2.) The condition of $\|\cdot\|_\TV$-boundedness also appears in \cite{Canizo}. It cannot be deduced generally from the assumed continuity of $\mu_\bullet$. Although $C:=\{\mu_t\;|\; t\in [0,T]\}\subset\CM([0,1])_\BL$ is compact, this does not imply that $\sup_{0\leqs t\leqs T}\|\mu_t\|_\TV<\infty$, unless $C\subset\CM^+([0,1])$. One can prove (for a Polish state space $S$), that if for every $K\subset \CM(S)_\BL$ compact, $\sup_{\mu\in K}\|\mu\|_\TV<\infty$, then $(\CM(S),\|\cdot\|_\TV)$ is linearly isomorphic to $\CMc(S)_\BL$, using \cite{Hille:2005}, Proposition 3.2. This is `rarely' the case, see \cite{Sander}, Theorem 3.11.\\
3.) Integral equation \eqref{eq:VOC} appears naturally from a probabilistic description of the system, see Section \ref{sec:probabilistic}.

\section{Well-posedness for a discontinuous perturbation}
\label{sec:solution limit eq}
\label{sec:limit well-postness ad hoc}
% 2 labels ???

In this section we consider the problem of existence, uniqueness and continuous dependence on initial conditions for measure-valued mild solutions to \eqref{eq:main equation}. Due to failure of (Lipschitz) continuity of $F_f$ the standard arguments using Picard iterations and Gronwall's Lemma cannot be used.

\begin{proposition}[Uniqueness]\label{lemma:uniqueness}
Equation \eqref{eq:VOC} has at most one measure-valued mild solution.
\end{proposition}

\begin{proof}
Let $\mu_\bullet$ and $\hat{\mu}_\bullet\in C([0,T],\CM([0,1])_\BL)$ be measure-valued mild solutions to \eqref{eq:VOC}. Then $\|\mu_t\|_\TV$ and $\|\hat{\mu}\|_\TV$ are bounded on $[0,T]$. According to Proposition \ref{prop:TV-estimates},
the function $s\mapsto\|\mu_s-\hat{\mu}_s\|_\TV$ is measureable. Moreover, it is bounded on $[0,T]$ by assumption, hence $L^1$. Then again by Proposition \ref{prop:TV-estimates},
\begin{equation}\label{eq:appl Gronwall}
\|\mu_t-\hat{\mu}_t\|_\TV \leqs \int_0^t \bigl\| P_{t-s}[F_f(\mu_s) - F_f(\hat{\mu}_s)]\bigr\|_\TV\,ds \leqs \|f\|_\infty\cdot \int_0^t \|\mu_s-\hat{\mu}_s\|_\TV\,ds.
\end{equation}
Grownwall's Lemma yields $\|\mu_t-\hat{\mu}_t\|_\TV=0$ for all $0\leqs t\leqs T$.
\end{proof}

\begin{corollary}
There exists at most one mild solution to \eqref{eq:VOC} in $C([0,T],\CM^+([0,1])_\BL)$.
\end{corollary}

\begin{proof}
The set $\{\mu_t\;|\; t\in [0,T]\}$ is compact in $\CM^+([0,1])_\BL$, hence $\|\cdot\|^*_\BL$-bounded. Now recall that $\|\mu\|^*_\BL=\|\mu\|_\TV$ for positive measures $\mu$.
\end{proof}

\noindent {\bf Remark.}\ 1.) In applications to e.g. crowd or population dynamics only positive solutions are interpretable, hence of major interest.\\
2.) Technically, for the application of Proposition \ref{prop:TV-estimates} in the proof of Proposition \ref{lemma:uniqueness} it is only required that $t\mapsto \|\mu_t\|_\TV$ is an $L^1$-function. However, if this seemingly weaker condition than $\|\cdot\|_\TV$-boundedness holds, then an argument similar to \eqref{eq:appl Gronwall} and application of Gronwall's Lemma yields that $t\mapsto \|\mu_t\|_\TV$ must be bounded on $[0,T]$.
\vskip 0.2cm

Since there is no `smoothing effect' in the dynamics in the interior of the interval $[0,1]$, we expect Dirac masses to stay Dirac masses. These move according to $P_t$. The latter acts simply on Dirac masses: $P_t\delta_x= \delta_{\Phi_t(x)}$. Therefore one may try as particular solution to \eqref{eq:VOC} with $\mu_0=\eps_x(0)\delta_x$:
\begin{equation}\label{eq:particular indiv solution}
\mu_t = \eps_x(t)\delta_{\Phi_t(x)}.
\end{equation}
Substitution of \eqref{eq:particular indiv solution} into \eqref{eq:VOC} yields, after evaluation of the measures on the Borel set $\{\Phi_t(x)\}$ and using \eqref{eq:set within integral}:
\begin{equation}\label{eq:VOC elementary solution}
\eps_x(t) = \eps_x(0) + \int_0^t \eps_x(s)f(\Phi_s(x))\, ds.
\end{equation}
Equation \eqref{eq:VOC elementary solution} is solved by the continuous (Lipschitz) function
\begin{equation}
\eps_x(t) = \eps_x(0)\exp\bigl( \int_0^t f(\Phi_s(x))\,ds\bigr) = \eps_x(0) \exp(f^\Phi_t(x))
\end{equation}
(cf. Lemma \ref{lem:crucial BL-function}).

Any initial measure $\mu_0$ is a superposition of Dirac masses, according to \eqref{eq:superposition of Diracs}. Therefore we obtain the following existence result and integral representation for the unique globally existing mild solution to \eqref{eq:VOC}:

\begin{proposition}[Existence]\label{prop:existence}
Let $f:[0,1]\to\R$ be a piecewise bounded Lipschitz function such that $v(x)\neq 0$ at any point $x$ of discontinuity of $f$.
Then for each $\mu_0\in\CM([0,1])$ there exists a continuous and locally $\|\cdot\|_\TV$-bounded solution $\mu_\bullet:\Rp\to\CM([0,1])_\BL$ to \eqref{eq:VOC} defined by
\begin{equation}\label{eq:semi-explicit solution}
\mu_t := \int_{[0,1]} \exp\bigl( \int_0^t f(\Phi_s(x))\,ds\bigr)\cdot \delta_{\Phi_t(x)}\, d\mu_0(x) = \int_{[0,1]} \exp(f^\Phi_t(x))\cdot \delta_{\Phi_t(x)}\, d\mu_0(x)
\end{equation}
as Bochner integral in $\CMc([0,1])_\BL$. It satisfies $\|\mu_t\|_\TV\leqs e^{\|f\|_\infty t}\|\mu_0\|_\TV$. Moreover, $\mu_\bullet$ is locally Lipschitz:
\begin{equation}\label{eq:Lipschitz estimate semi-explicit sol}
\|\mu_t-\mu_{t'}\|_\BL^* \leqs \|\mu_0\|_\TV\cdot \bigl(\|f\|_\infty + \|v\|_\infty\bigr)\cdot e^{\|f\|_\infty\max(t,t')}\cdot |t-t'|.
\end{equation}
\end{proposition}

\begin{proof}
The integrand in \eqref{eq:semi-explicit solution} is a bounded continuous function from $[0,1]$ into $\CM^+([0,1])_\BL$ (Lemma \ref{lem:crucial BL-function}). Thus for $\mu_0\in\CM^+([0,1])$ the Bochner integral exists, with value in $\CM^+([0,1])$, because this cone is closed. For $\mu_0\in\CM([0,1])$ the integral yields a {\em measure} in $\CM([0,1])\subset \CMc([0,1])_\BL$, by using the Jordan decomposition $\mu_0=\mu_0^+ -\mu^-_0$. So $\mu_t\in\CM([0,1])$ for all $t$.

Next we prove that $\mu_t$ defined by \eqref{eq:semi-explicit solution} satisfies \eqref{eq:Lipschitz estimate semi-explicit sol}. So let $t,t'\in\Rp$. We may assume $t>t'$. Then
\begin{align*}
\|\mu_t-\mu_{t'}\|_\BL^* & \leqs \int_{[0,1]} \bigl\| \exp(f^\Phi_t(x))\cdot\delta_{\Phi_t(x)} - \exp(f^\Phi_{t'}(x))\cdot \delta_{\Phi_{t'}(x)} \bigr\|_\BL^*\, d|\mu_0|(x)\\
  & \leqs \int_{[0,1]} \exp(f^\Phi_t(x)) \bigl\| \delta_{\Phi_t(x)} - \delta_{\Phi_{t'}(x)} \bigr\|_\BL^*\, d|\mu_0|(x)\\
	& \qquad + \int_{[0,1]} \bigl| \exp(f^\Phi_t(x)) - \exp(f^\Phi_{t'}(x)) \bigr| \,d|\mu_0|(x).
\end{align*}
According to \eqref{eq:Lipschitz prop If},
\[
\bigl| \exp(f^\Phi_t(x)) - \exp(f^\Phi_{t'}(x)) \bigr| \leqs \|f\|_\infty e^{\|f\|_\infty \max(t,t')}\cdot
|t-t'|.
\]
Using \eqref{eq:Bl-norm difference Diracs} and Lemma \ref{lem:basic props indiv flow}({\it ii}\,),
\[
\bigl\| \delta_{\Phi_t(x)} - \delta_{\Phi_{t'}(x)} \bigr\|_\BL^* \leqs d(\Phi_t(x),\Phi_{t'}(x)) \leqs \|v\|_\infty\cdot |t-t'|.
\]
Now \eqref{eq:Lipschitz estimate semi-explicit sol} simply follows.
\end{proof}

\begin{corollary}
For every $\mu_0\in\CM^+([0,1])$ there exists a unique mild solution to \eqref{eq:VOC} that is in $\CM^+([0,1])$ for all time.
\end{corollary}

The classical argument with Gronwall's Inquality to obtain continuous dependence on initial conditions fails in this setting, because the pertubation is not Lipschitz continuous. Instead we use the crucial observation made in Lemma \ref{lem:crucial BL-function}.

\begin{proposition}[Continuous dependence on initial conditions]\label{prop:cont dependence} Assume that $f:[0,1]\to\R$ is a piecewise bounded Lipschitz function such that $v(x)\neq 0$ at any point $x$ of discontinuity of $f$. Then for each $T\geqs 0$, there exists $C_T>0$ such that for all initial values $\mu_0,\mu'_0\in\CM([0,1])$ the corresponding $\|\cdot\|_\TV$-bounded mild solutions $\mu_\bullet$ and $\mu'_\bullet$ to \eqref{eq:VOC} satisfy
\begin{equation}\label{eq:cont dep ininitial cond}
\|\mu_t-\mu'_t\|_\BL^* \leqs C_T \|\mu_0-\mu'_0\|_\BL^*
\end{equation}
for all $t\in[0,T]$.
\end{proposition}

\begin{proof}
Let $\phi\in\BL([0,1])$ and $t\geqs 0$. According to Lemma \ref{lem:crucial BL-function}, $f^\Phi_t$ is a bounded Lipschitz function on $[0,1]$. Hence so is $x\mapsto \exp(f_t^\Phi(x))$. According to the representation \eqref{eq:semi-explicit solution}, we have
\[
|\pair{\mu_t-\mu'_t}{\phi}| = |\pair{\mu_0-\mu'_0}{(\phi\circ\Phi_t)\cdot (\exp\circ f^\Phi_t)}|.
\]
Because $(\BL([0,1]),\|\cdot\|_\BL)$ is a Banach algebra for pointwise multiplication of functions,
\begin{align*}
|\pair{\mu_t-\mu'_t}{\phi}| &\leqs \|\mu_0-\mu'_0\|_\BL^*\cdot \|\phi\circ\Phi_t\|_\BL \|\exp\circ f^\Phi_t\|_\BL\\
  & \leqs \|\mu_0-\mu'_0\|_\BL^*\cdot \max(1,|\Phi_t|_L)\|\phi\|_\BL\cdot e^{\|f\|_\infty t}(1+ |f^\Phi_t|_L).
\end{align*}
By taking the supremum over $\phi$ in the unit bal of $\BL([0,1])$ and using Lemma \ref{lem:basic props indiv flow} {\it (i)} and \eqref{eq:upper bound Lipschitz int f} there exists $C_T<\infty$ such that \eqref{eq:cont dep ininitial cond} holds for all $0\leqs t\leqs T$.
\end{proof}

\section{Approximation by regularisation}
\label{sec:regularized systems}

Equation \eqref{eq:VoC} involving the piecewise bounded Lipschitz function $f$ can be considered (formally) as a limit of a sequence of equations with $f$ replaced by $f_n\in\BL([0,1])$, a `regularisation' of $f$, such that $f_n\to f$ as $n\to\infty$ in suitable sense, e.g. point-wise. Here, we investigate how mild solutions $\mu^\nn_\bullet$ to \eqref{eq:VoC} for $f_n$ relate to the solution $\mu_\bullet$ associated to the limiting piecewise bounded Lipschitz $f$ discussed in the previous section.

Throughout this section, fix the bounded Lipschitz velocity field $v$ on $[0,1]$ and let $f$ be a piecewise bounded Lipschitz function on $[0,1]$ such that the set of discontinuities of $f$ is disjoint from the set of zeros of $v$. Let $f_n\in\BL([0,1])$ such that $f_n\to f$ point-wise on $[0,1]$. Let $\mu_0\in\CM([0,1])$ and let $\mu^\nn_\bullet$ be the unique mild solution to \eqref{eq:VoC} associated to $f_n$ and $\mu_\bullet$ the solution associated to $f$, both with initial value $\mu_0$.

\begin{lemma}\label{lem:convergence of approx}
Suppose there exists $M\geqs 0$ such that $\|f_n\|_\infty\leqs M$ for all $n$. Then
\begin{equation}\label{eq:estimate approx regular 1}
\|\mu^\nn_t-\mu_t\|_\BL^* \leqs e^{Mt} \int_{[0,1]}\int_0^t \bigl| f_n(\Phi_s(x))-f(\Phi_s(x))\bigr|\, ds\, d|\mu_0|(x).
\end{equation}
In particular, $\mu^\nn_t$ converges to $\mu_t$ in $\CM([0,1])_\BL$ as $n\to\infty$ for every $t\geqs 0$.
\end{lemma}
\begin{proof}
First observe that $\|f\|_\infty\leqs M$. Let $\phi\in\BL([0,1])$. Using the representation \eqref{eq:semi-explicit solution} for the mild solutions $\mu^\nn_\bullet$ and $\mu_\bullet$, we obtain
\begin{align*}
\left|\pair{\mu^\nn_t-\mu_t}{\phi}\right| & \leqs \int_{[0,1]} \bigl| \exp\bigl( (f_n)^\Phi_t(x) \bigr) - \exp\bigl( f^\Phi_t(x) \bigr) \bigr| \cdot |\phi(\Phi_t(x))|\, d|\mu_0|(x)\\
& \leqs \|\phi\|_\infty e^{Mt} \int_{[0,1]} \bigl| (f_n)^\Phi_t(x)  - f^\Phi_t(x) \bigr| d|\mu_0|(x).
\end{align*}
This gives \eqref{eq:estimate approx regular 1} by taking the supremum over $\phi$ in the unit ball.
Using this equation, Lebesgue Dominated Convergence Theorem yields the last statement, because of  the point-wise convergence of $f_n$ to $f$ and the assumed uniform upper bound on all $f_n$.
\end{proof}

Lemma \ref{lem:convergence of approx} and inspection of the double integral in equation \eqref{eq:semi-explicit solution} in particular indicate that further properties of the convergence of $\mu^\nn_\bullet$ to $\mu_\bullet$ depend in a delicate way on the interplay between the way the approximating sequence $(f_n)$ relates to $f$, the initial condition and properties of the flow $(\Phi_t)_{t\geqs 0}$, such as uniformity of convergence on compact time intervals or the rate of convergence. We show now that a particular type of regularisation of $f$ provides approximating sequences $(f_n)\subset \BL([0,1])$ that yield uniform convergence on compact intervals together with an upper bound for the rate of convergence that works for any initial condition in $\CM([0,1])$.

Observe that there exist $f_n\in\BL([0,1])$ such that $f_n\to f$ point-wise and the sets
\[
\Delta_n := \{x\in[0,1]\;|; f_n(x)\neq f(x)\}
\]
are such that $|\Delta_n|\to 0$ as $n\to\infty$, where $|\cdot|$ denotes Lebesgue measure of the set.
Because the set of discontinuities of $f$ is disjoint from the set of zeros for $v$, there exists such approximating sequences such that each difference set $\Delta_n$ is a union of finitely many open intervals where each interval contains precisely one point of discontinuity of $f$ and the velocity field $v$ is bounded away from zero on this interval, uniformly in $n$. Moreover, if $\sup_n\|f_n\|_\infty<\infty$, then we call such a sequence $(f_n)$ a {\em regularly approximating sequence}.

\begin{proposition}\label{prop:rate of convergence}
Let $(f_n)$ be a regularly approximating sequence for $f$, with difference sets $\Delta_n$ and uniform upper bound $M:=\sup_n\|f_n\|_\infty$. Then there exists $C>0$ such that
\begin{equation}
\|\mu^\nn_t-\mu_t\|_\BL^* \leqs 2CMe^{Mt}\cdot |\Delta_n|.
\end{equation}
In particular, $\mu^\nn_t$ converges to $\mu_t$ uniformly on compact time intervals, at the same rate as $|\Delta_n|\to 0$ when $n\to \infty$.
\end{proposition}
\begin{proof}
Starting from \eqref{eq:estimate approx regular 1} we obtain
\begin{align}
\|\mu^\nn_t-\mu_t\|_\BL^* &\leqs e^{Mt} \int_{[0,1]} \int_0^t\bigl| f_n(\Phi_s(x))-f(\Phi_s(x))\bigr|\cdot\ind_{\Delta_n}(\Phi_s(x))\, ds \, d|\mu_0|(x)\nonumber\\
&\leqs e^{Mt}\cdot 2M\int_{[0,1]} \int_0^t\ind_{\Delta_n}(\Phi_s(x))\, ds \, d|\mu_0|(x).\label{eq:rate of conv inner int}
\end{align}
The inner integral in \eqref{eq:rate of conv inner int} is the time that the orbit under the flow $(\Phi_t)_{t\geqs 0}$ starting at $x$ spends in $\Delta_n$. It is bounded above by $(\inf_{x\in\Delta_n} |v(x)|)^{-1}|\Delta_n|$. Since $(f_n)$ is a regularly approximating sequence, there exists $C>0$ such that
\[
\sup_n\, \bigl(\inf_{x\in\Delta_n} |v(x)|\bigr)^{-1} \leqs C.
\]
Thus,
\begin{equation}
\|\mu^\nn_t-\mu_t\|_\BL^* \leqs 2Me^{Mt}\cdot\|\mu_0\|_\TV\cdot C |\Delta_n|.
\end{equation}
The statement on uniform convergence and rate of convergence immediately follows.
\end{proof}
\vskip 0.2cm

\noindent{\bf Example.}
Consider the situation where the velocity field $v\in\BL([0,1])$ satisfies $v(1)>0$ and $f(x)=-a\ind_{\{1\}}$, with $a>0$. That is, in the interior of the unit interval no mass is removed or added, while mass that has accumulated at the boundary point 1 is removed at a rate $a$. $f$ is discontinuous, but piecewise bounded Lipschitz, clearly. A sequence of regularizers for $f$ may be defined by the sequence $(f_n)\subset\BL([0,1])$ given by $f_n(x):=-a[n(x-(1-\frac1n))]^+$, where $[\,\cdot\,]^+$ denotes the positive part of the argument. That is,
\begin{equation}\label{eqn: fn example}
f_n(x):=
\left\{
  \begin{array}{ll}
    0, & \quad \mbox{for}\ \hbox{$x\in[0,1-\frac1n)$;} \\
    -an(x-(1-\frac1n)), & \quad \mbox{for}\ \hbox{$x\in[1-\frac1n,1]$.}
  \end{array}
\right.
\end{equation}
It is easy to see that $-a\leqs f_n(x)\leqs 0$, so $\|f_n\|_\infty\leqs a$ and
\begin{equation}\label{eq:BL-norm fn}
\|f_n\|_{\BL}=\|f_n\|_\infty+|f_n|_L = a(1+n).
\end{equation}
Moreover, $\Delta_n = (1-\frac1n, 1]$ and $|\Delta_n|=\frac1n$. Note that $(1-\frac1n, 1]$ is \textit{open} in $[0,1]$. When $n$ is taken sufficiently large, $v$ will be bounded away from zero on $\Delta_n$, zo $(f_n)$ as defined above is a regularly approximating sequence for $f$. In effect, it realizes a family of systems in which there is a small boundary layer near 1, of thickness $\frac1n$, in which already mass is removed at rate $a$.

Proposition \ref{prop:rate of convergence} predicts that the solutions $\mu_\bullet^\nn$ corresponding to the regularized perturbation $f_n\in\BL([0,1])$ converge to the solution $\mu_\bullet$ corresponding to the discontinuous perturbation $f$ at rate $\mathcal{O}(1/n)$ for the norm $\|\cdot\|_\BL^*$. The role of the next section is to support this convergence rate also numerically.

\section{Numerical approximations}
\label{section: numerics}

%We consider two numerical approaches.
We consider absolutely continuous and discrete initial data, leading \textit{grosso modo} to an evolution in terms of a classical PDE and a system of ODEs, respectively. For the absorption of mass we consider both the regularization (boundary layer) and the limit process. Ultimately, we wish to verify the order of convergence stated in Proposition \ref{prop:rate of convergence} and the subsequent example.\\
Absorption is prescribed by $f_n$ as defined in \eqref{eqn: fn example}. For brevity and convenience, in the sequel we assume the limit case to be incorporated in this notation, i.e. we allow for $n=\infty$ and define  then $f_\infty:=f=-a\,\ind_{\{1\}}$. As before $v$ is bounded Lipschitz and for simplicity we let it satisfy $v(0)>0$ and $v(1)>0$. In fact, we take $v>0$ everywhere in this section.

\subsection{Two models}
\label{section: two approaches numerics}

Let $\rho_0:[0,1]\to\Rp$ be such that $\int_0^1\rho_0(x)\,dx=1$ and define the associated measure $\bar{\mu}_0$ by $d\bar{\mu}_0:=\rho_0 d\lambda$. Note that $\bar{\mu}_0$ is a probability measure on $[0,1]$, by definition of $\rho_0$.\\
\\
(1) We consider our model equation \eqref{eq:VOC} completed with initial data $\mu_0=\bar{\mu}_0$. Because of the choice of $\rho_0$, the model simplifies to a linear transport equation for the density $\rho$. Due to the definition of the individualistic flow \eqref{individualistic flow Phi}, mass accumulates at $x=1$, as $v$ points outward there. We keep track of the mass at $x=1$, a quantity called $\nu=\nu(t)$, the time-derivative of which is related to the flux of $\rho$ at $x=1$. For some $n\in\Np\cup\{\infty\}$ we solve
\begin{equation}\label{eq:main PDE}
\left\{
  \begin{array}{ll}
    \dfrac{\partial \rho}{\partial t} + \dfrac{\partial}{\partial x}(\rho v) = f_n \rho, & \hbox{on $(0,1)$,} \\
    \rho(0,\cdot)=\rho_0, & \hbox{}\\
    \dfrac{d\nu}{dt}=\rho(t,1)\,v(1)-a\nu(t). & \hbox{}
  \end{array}
\right.
\end{equation}
Since $v>0$ on $[0,1]$, we use an upwind scheme. At $x=0$, we need to provide a boundary condition. The definition \eqref{Def Pt push forward} of $(P_t)_{t\geqslant0}$ by means of a push-forward mapping suggests that there should be no influx of mass at $x=0$.
%(N.B. By $v>0$, mass is forced to move to the right, and there is no mass outside $[0,1]$. In fact, in this setting there is no such thing as `outside $[0,1]$'.)\\
The boundary condition is therefore $\rho(t,0)v(0)\leqslant 0$ which simplifies, for strictly positive $v$, to $\rho(t,0)=0$ for all times $t\geqslant0$.\\
\\
(2) We consider a particle system that is the discrete counterpart of \eqref{eq:main PDE}. Note that both models are instances of the general measure-valued model \eqref{eq:VOC} in their own right. Given a discrete initial measure, we know that the solution is a discrete measure for all time (see Section \ref{sec:limit well-postness ad hoc}, Proposition \ref{prop:existence} in particular). % (cf.~Lemma \ref{lemma ordering diracs} for the regularized case).
Once we have provided an initial measure $\mu_0=\sum_{i=1}^N\alpha_i(0)\delta_{x_i(0)}$ (for some fixed $N$), we  search for solutions $\mu=\sum_{i=1}^N\alpha_i(t)\delta_{x_i(t)}$.   Particularly, we take $\alpha_i(0)=1/N$ for all $i=1,\ldots,N$, while $\{x_i(0)\}_{i=1}^N$ consists of $N$ independent random positions, distributed according to $\bar{\mu}_0$. The evolution of the positions $\{x_i(t)\}_{i=1}^N$ is deterministic, dictated by $\Phi_t$ \eqref{individualistic flow Phi}. The masses $\alpha_i$ satisfy
\begin{equation}
\dfrac{d\alpha_i}{dt}=f_n(x_i)\,\alpha_i,
\end{equation}
for all $i\in\{1,\ldots,N\}$, where $n\in\Np\cup\{\infty\}$.\\
\\
The particles move without interactions, which implies that the numerics are relatively `cheap'. To trace the evolution of the particles, we use a forward Euler scheme.

\subsection{Evolution of mass within $[0,1)$ and at $1$}\label{subsection num results}
We specify
\begin{equation}
\rho_0(x):=
\left\{
  \begin{array}{ll}
    4x, & \hbox{$0\leqslant x\leqslant \dfrac12$,} \\
    4-4x, & \hbox{$\dfrac12< x\leqslant 1$,}
  \end{array}
\right.
\end{equation}
and
\begin{equation}
v(x) :=  \dfrac34 + 2(x-\dfrac12)^2 \hspace{1cm}\text{for all }x\in[0,1].
\end{equation}
Moreover, we take $a=1/2$, and $N=25000$ particles.\\
\\
We first compare the solutions for the absolutely continuous and discrete initial measures described in Section \ref{section: two approaches numerics}. For the discrete measure, we derive an approximate density by splitting $[0,1)$ (excluding $1$) in 100 intervals, and dividing the total mass in any interval by its length. In our graphs this associated density is indicated by `av.', since mass is averaged over space.\\
\\
See Figure \ref{fig:limit} for the time evolution of the limit case (vanished boundary layer, ``$n=\infty$"). Some features must be noted. First of all, we observe a deformation of the initial density profile as time proceeds, due to the fact that $v$ is not constant. Stretching and compression is a direct consequence of the change in monotonicity of $v$. Roughly, the two solutions depict the same behaviour. Oscillations are inherent to the nature of a particle system. As time proceeds, the discrete model deviates more from the PDE. The reason is that particles have accumulated in $1$, leaving fewer particles in the interior, and thus leading to a coarser approximation.\\
\\
\begin{figure}
        \centering
        \begin{subfigure}[b]{0.45\textwidth}
                \includegraphics[width=\textwidth]{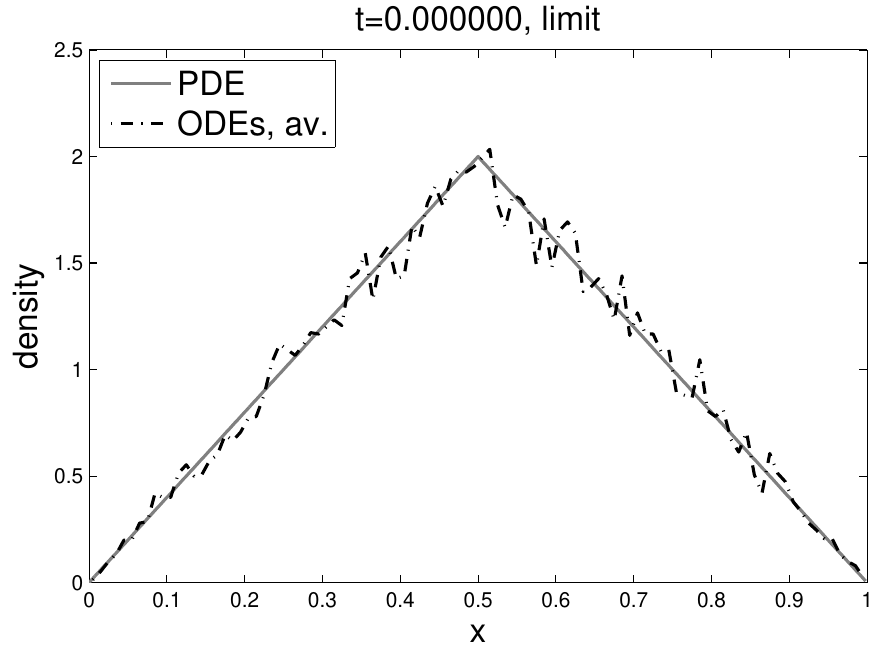}
                \caption{$t=0$}
                \label{fig:limit 0}
        \end{subfigure}%
        ~ %add desired spacing between images, e. g. ~, \quad, \qquad etc.
        \begin{subfigure}[b]{0.45\textwidth}
                \includegraphics[width=\textwidth]{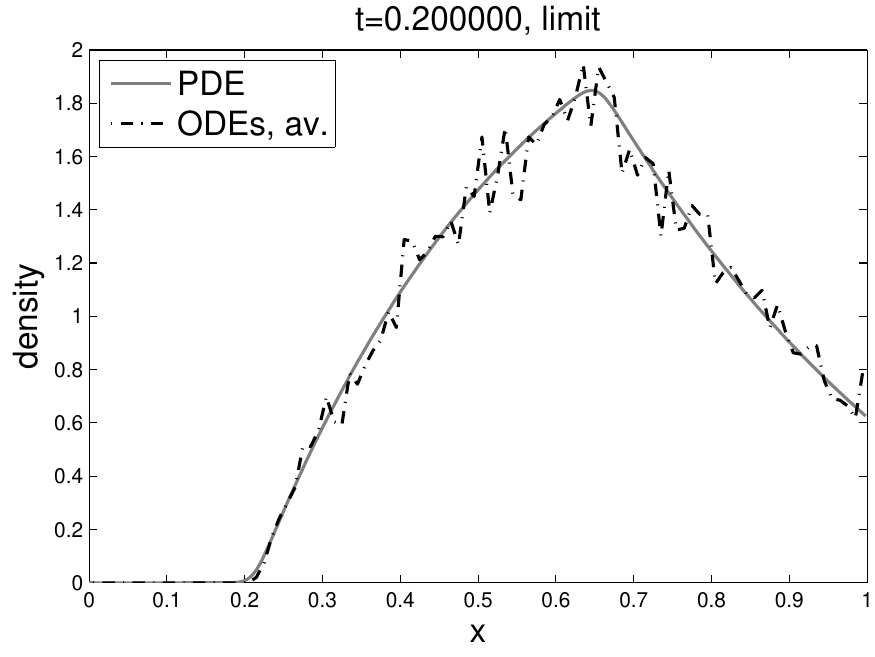}
                \caption{$t=0.2$}
                \label{fig:limit 0.2}
        \end{subfigure}\\
        \vspace{0.5 cm}
        \begin{subfigure}[b]{0.45\textwidth}
                \includegraphics[width=\textwidth]{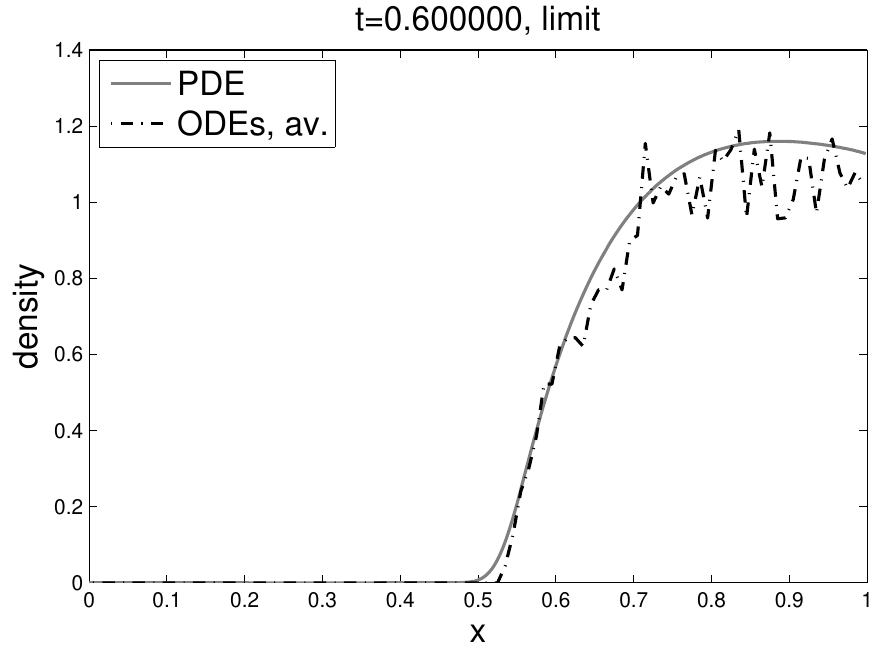}
                \caption{$t=0.6$}
                \label{fig:limit 0.6}
        \end{subfigure}
        \begin{subfigure}[b]{0.45\textwidth}
                \includegraphics[width=\textwidth]{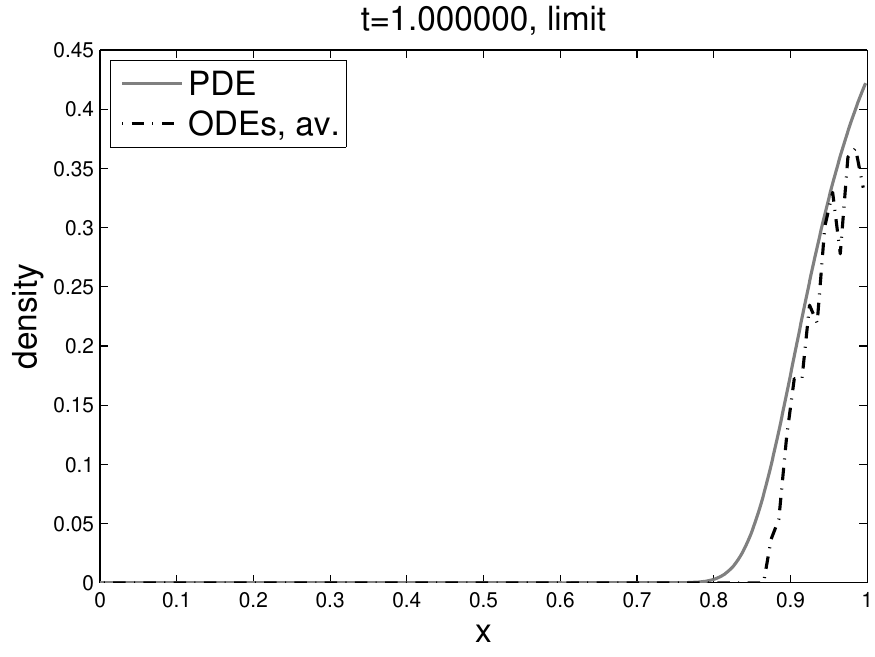}
                \caption{$t=1$}
                \label{fig:limit 1}
        \end{subfigure}
        \caption{The limit case: Comparison of absolutely continuous (`PDE') and discrete (`ODEs, av.') solutions in the interior of the domain.}\label{fig:limit}
\end{figure}
In Figure \ref{fig:n2} we show the results corresponding to the regularized system for $n=2$. We observe the same behaviour as in Figure \ref{fig:limit}. Note that the densities in Figure \ref{fig:n2} are slightly smaller than in Figure \ref{fig:limit}; in the regularized system mass already decays in the interior of the domain. The differences between the regularized and limit systems being small, reflect the fact that, apparently, the magnitude of $v$ is large compared to the rate at which mass decays. Mass arrives at $1$ due to $v$ relatively fast, without having too much chance to decay.\\
\\
\begin{figure}
        \centering
        \begin{subfigure}[b]{0.45\textwidth}
                \includegraphics[width=\textwidth]{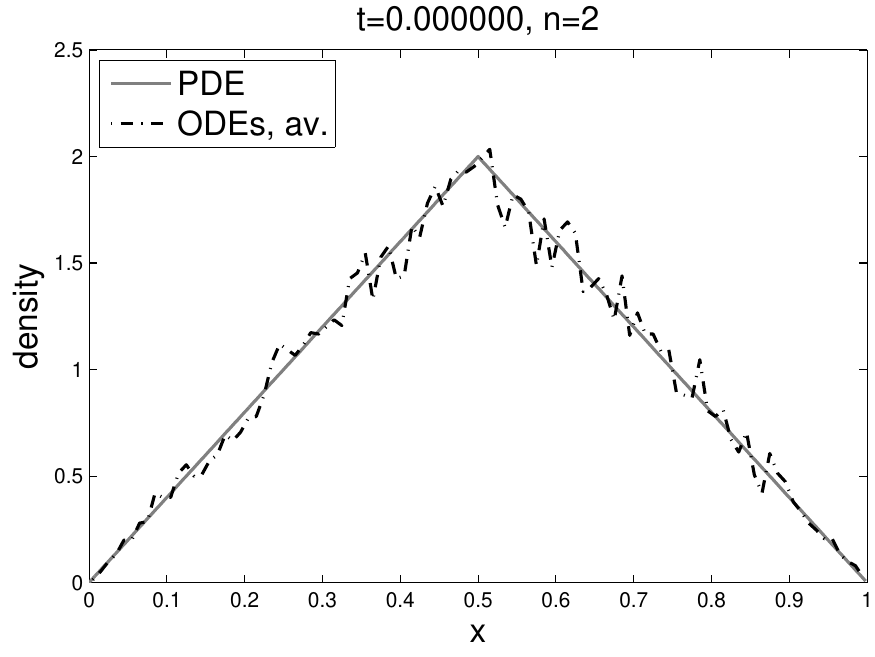}
                \caption{$t=0$}
                \label{fig:n2 0}
        \end{subfigure}%
        ~ %add desired spacing between images, e. g. ~, \quad, \qquad etc.
        \begin{subfigure}[b]{0.45\textwidth}
                \includegraphics[width=\textwidth]{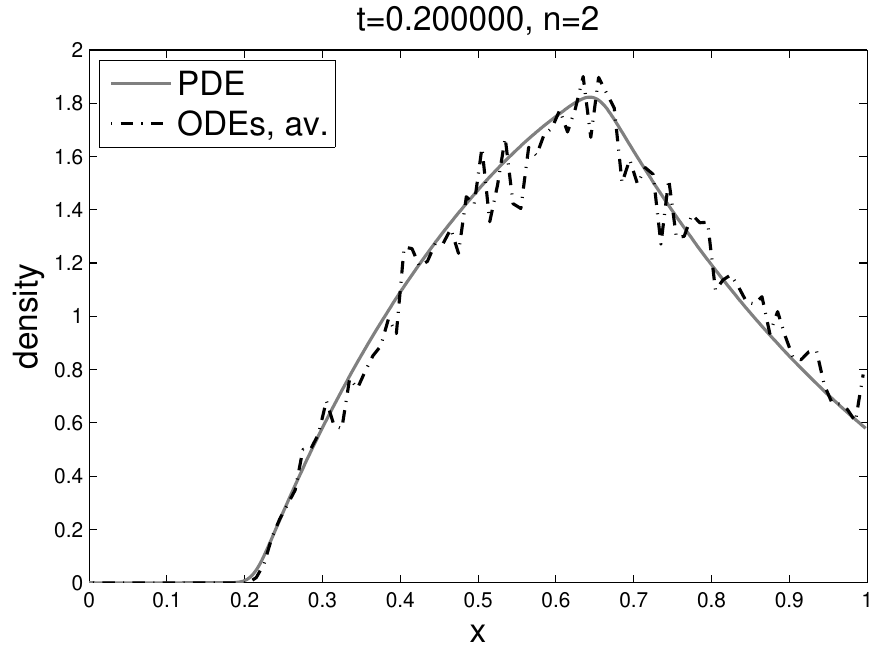}
                \caption{$t=0.2$}
                \label{fig:n2 0.2}
        \end{subfigure}\\
        \vspace{0.5 cm}
        \begin{subfigure}[b]{0.45\textwidth}
                \includegraphics[width=\textwidth]{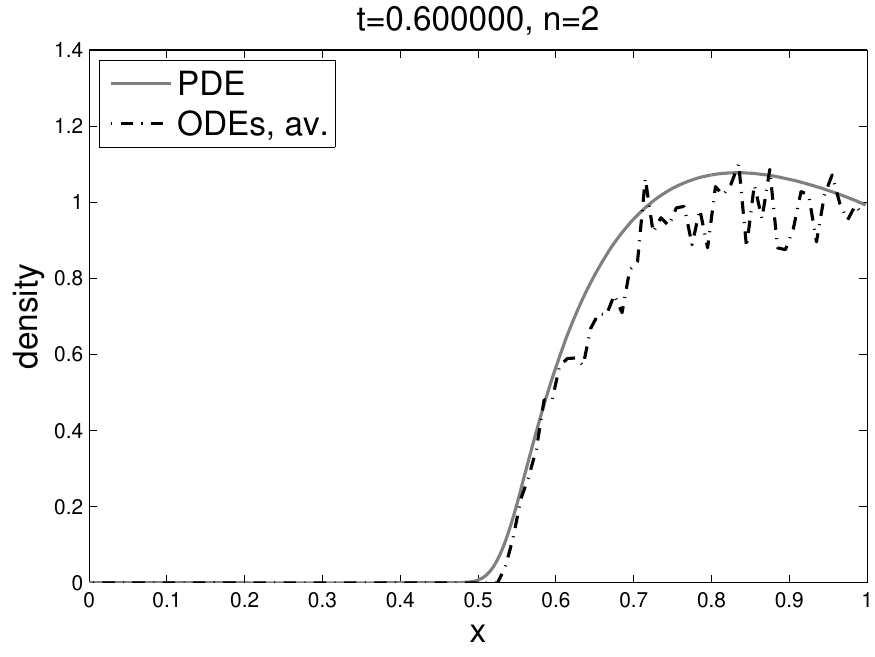}
                \caption{$t=0.6$}
                \label{fig:n2 0.6}
        \end{subfigure}
        \begin{subfigure}[b]{0.45\textwidth}
                \includegraphics[width=\textwidth]{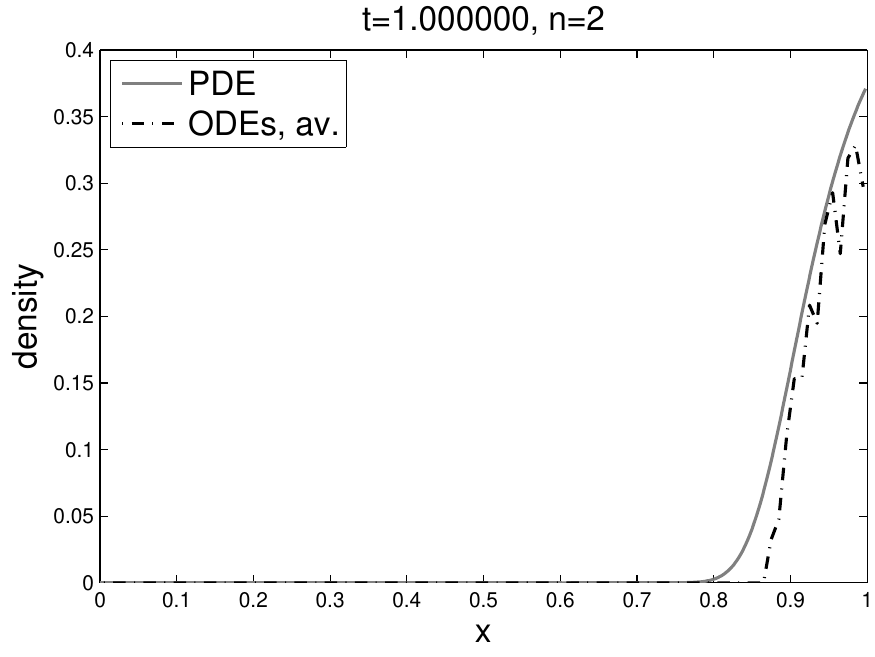}
                \caption{$t=1$}
                \label{fig:n2 1}
        \end{subfigure}
        \caption{For the regularised case ($n=2$): Comparison of absolutely continuous (`PDE') and discrete (`ODEs, av.') solutions in the interior of the domain.}\label{fig:n2}
\end{figure}
In Figure \ref{fig:at one limit n2} we compare the evolution of mass accumulation and decay at $x=1$. Both for the limit case and the case $n=2$ the solutions for absolutely continuous and discrete initial data are nearly indistinguishable. It is worth noting  that in the regularized case the peak value is smaller than in the limit case. This is in agreement with our observations above about mass distribution in the interior of the domain. Some mass is taken away in the boundary layer before arriving at $x=1$.\\
\\
\begin{figure}
        \centering
        \begin{subfigure}[b]{0.45\textwidth}
                \includegraphics[width=\textwidth]{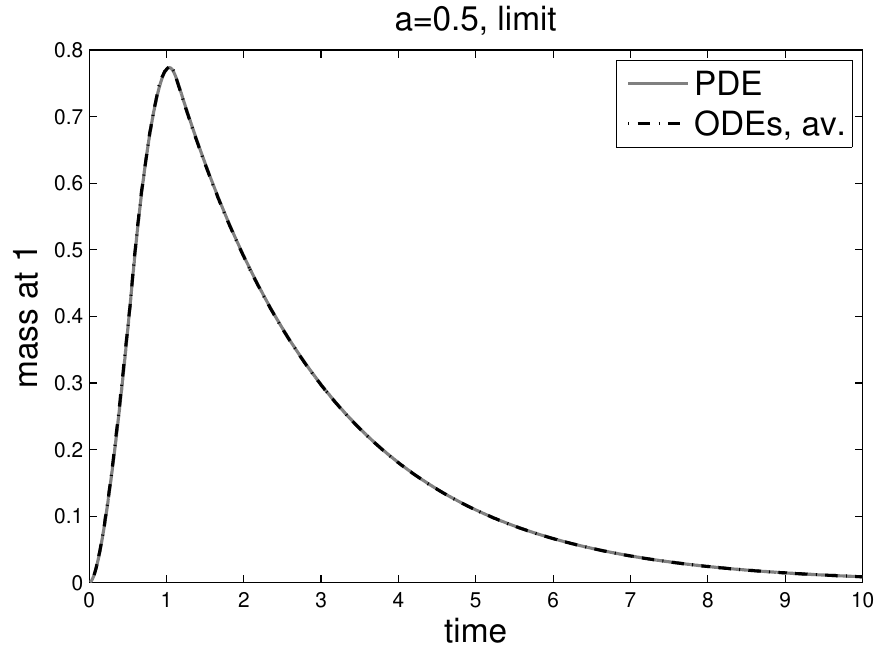}
                \caption{Limit}
                \label{fig:limit at one}
        \end{subfigure}%
        ~ %add desired spacing between images, e. g. ~, \quad, \qquad etc.
        \begin{subfigure}[b]{0.45\textwidth}
                \includegraphics[width=\textwidth]{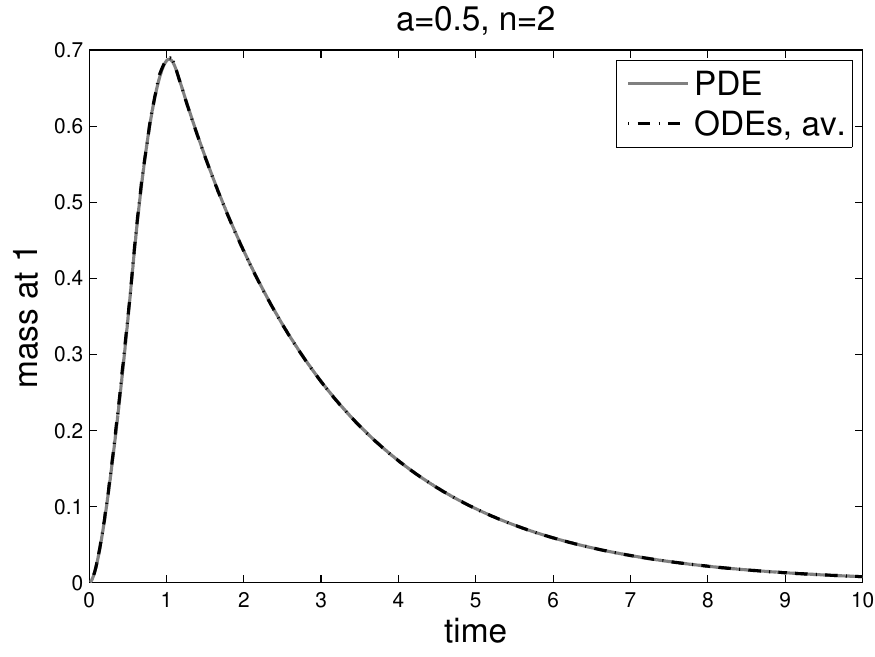}
                \caption{Regularized, $n=2$}
                \label{fig:n2 at one}
        \end{subfigure}
        \caption{Comparison of the mass at $x=1$ for the absolutely continuous (`PDE') and discrete (`ODEs, av.') solutions.}\label{fig:at one limit n2}
\end{figure}
As $n$ increases, we observe similar behaviour as in the case $n=2$, be it that the solution of the limit problem is approximated even better. We omit to show the corresponding graphs.
%We stop comparing the solutions for absolutely continuous and discrete initial measures here. In the next section we verify the order of convergence of the regularization for discrete initial data only.

\subsection{Order of convergence}
Proposition \ref{prop:rate of convergence} (and the subsequent example) state that $\|\mu_t^{(n)}-\mu_t\|^*_{\BL} = \mathcal{O}(1/n)$, uniformly on compact time intervals. In this section, we confirm this statement numerically. We follow the idea of \cite{Jablonski} for calculating the flat metric. They choose to use the Fortet-Mourier norm $\|\cdot\|_{\FM}:=\max\{\left\|\cdot\|_{\infty},|\cdot|_L\right\}$ on the space of bounded Lipschitz functions, rather than the $\BL$ norm, or Dudley norm $\|\cdot\|_{\BL}:=\|\cdot\|_\infty+|\cdot|_L$ used in this paper. Consequently, the corresponding dual norm is different; $\|\cdot\|^*_{\FM}$ instead of $\|\cdot\|^*_{\BL}$. However, these norms are equivalent (see \eqref{eq:equivalence Dudely Fortet-Mourier}). So either one of them can be used for estimating the order of convergence.\\
\\
The algorithm in \cite{Jablonski} can only be applied to (signed) discrete measures. In the sequel we thus focus on discrete measures only. We have seen before that any discrete measure stays discrete as time evolves, which allows for the use of the algorithm in \cite{Jablonski}. For a comparison between the solutions for absolutely continuous and discrete initial data, we refer the reader back to Section \ref{subsection num results}.\\
\\
We take $n=2^k$ for $k=1,2,\ldots$ and define
\begin{equation}\label{T}
A_k:=\sup_{t\in[0,T]}\|\mu_t^{(2^k)}-\mu_t\|^*_{\FM}.
\end{equation}
In (\ref{T}), $T$ denotes the final time of the computation. We estimate the order of convergence $q$, that is the value of $q$ such that
\begin{equation}\label{Ak order conv q}
A_k = \mathcal{O}\left(\left(\dfrac{1}{2^k}\right)^q\right),
\end{equation}
as $k\rightarrow\infty$. As mentioned before, $q=1$ should hold. We use Richardson extrapolation and approximate the value of $q$ by $_2\log(A_k/A_{k+1})$. The results are in Table \ref{Table order of convergence} and support our theoretical claim.

\begin{table}
 \centering
\begin{tabular}{r|l}
  $k$ & $_2\log(A_k/A_{k+1})$ \\\hline
  1 & 1.0530718 \\
  2 & 1.0716763 \\
  3 & 1.0445858 \\
  4 & 1.0237701 \\
  5 & 1.0123283 \\
  6 & 1.0065000 \\
  7 & 1.0023709 \\
  8 & 1.0000000 \\
\end{tabular}
\caption{Estimation of the order of convergence $q$, as used in \eqref{Ak order conv q}, for $n=2^k$, $k\in\Np$. The numerical order of convergence is $\mathcal{O}(1/n)$, confirming the theoretical result from Proposition \ref{prop:rate of convergence}.}\label{Table order of convergence}
\end{table}

\section{A probabilistic interpretation of the integral equation}
\label{sec:probabilistic}

The measure-valued variation of constants formula \eqref{eq:VoC} derives naturally from a probabilistic view on the system, as we shall now describe.

Take $N$ individuals in a confined space, with position $X^i_t\in[0,1]$ at time $t$ say ($i=1,\dots, N$). We assume that the boundary at $1$ is sticking. By this we mean that at the absorbing boundary we have a `gate' that absorbs an individual present there a time $T_i$ after arrival, which is an exponentially distributed random variable with (constant) rate $a$. We assume that the individuals are indistinguishable and the absorption of individuals (gating) occurs independently. We denote by $\pi^{(i)}_t$ the law of $X^i_t$ when $X_0^i$ is distributed according to the probability measure $\pi_0$.

Since individuals are independent, the expected number of individuals in a Borel set $E\subset [0,1]$ is given by the measure $\mu_t(E)$ , where $\mu_t$ satisfies
\begin{equation}\label{eq:prob mu}
\mu_t(E) = \mathbb{E}\left[\sum_{i=1}^N \mathbbm 1_{X_t^i}(E) \right]\ =\ \sum_{i=1}^N\pi_t^{(i)}(E).
\end{equation}
The measures $\pi_t^{(i)}$ satisfy
\begin{equation}\label{S}
\pi_t(E)=P_t\pi_0(E)-\delta_1(E)\int_0^t a\pi_s\left(\{1\}\right)ds,
\end{equation}
such that \eqref{eq:prob mu} together with \eqref{S} yields \eqref{eq:VoC} with the particular choice $f=-a\ind_{\{1\}}$ as in Section \ref{sec:regularized systems}, Example.
To see that \eqref{S} holds, let us introduce the conditional probability
\begin{eqnarray}
p(t,\Delta t)&:=& \mathrm{Prob}\bigl(\mbox{Individual is gated in } [t,t+\Delta t] \  | \  X_t=1 \bigr)\nonumber\\
&=& 1-e^{-a\Delta t}.
\end{eqnarray}
Discretize the time interval $[0,t]$ into $\kappa$ steps of length $\Delta s_j$ and let $s_j$ be the left point of the $j$-th subinterval. Then the probability that the individual has been gated in $[0,t]$ is approximately
\begin{eqnarray}
\sum_{j=1}^\kappa p(s_j,\Delta s_j)\pi_{{s_j}}\left(\{1\}\right)&=& \sum_{j=1}^\kappa\frac{p(s_j,\Delta s_j)}{\Delta s_j}\pi_{s_j}\left(\{1\}\right) \Delta s_j\nonumber\\
&\to & \int_0^t a\pi_s\left(\{1\}\right) ds \mbox{ as } \kappa \to\infty.
\end{eqnarray}
Formulating now (\ref{S}) in terms of measures in $\mathcal{M}^+([0,1])_{\BL}$ gives
\begin{equation}\label{SM}
\pi_t=P_t\pi_0-\int_0^t a\pi_s\left(\{1\}\right)ds\,\delta_1.
\end{equation}

To conclude this section, we compare the numerics of Section \ref{section: numerics} to a numerical approximation of its probabilistic counterpart presented in this section. We only consider those results of Section \ref{section: numerics} corresponding to vanished boundary layer; cf.~the dashed curves in Figure \ref{fig:limit} and Figure \ref{fig:limit at one}. For the deterministic and stochastic models, the same discrete initial measure is used with positions drawn randomly from the distribution $\bar{\mu}_0$. Recall that we use $N=25000$ Diracs of initial mass $1/N$.

%%%%%%%%%%%%%%%%%%%%%%%%
\begin{figure}[h!]
        \centering
        \includegraphics[width=0.45\textwidth]{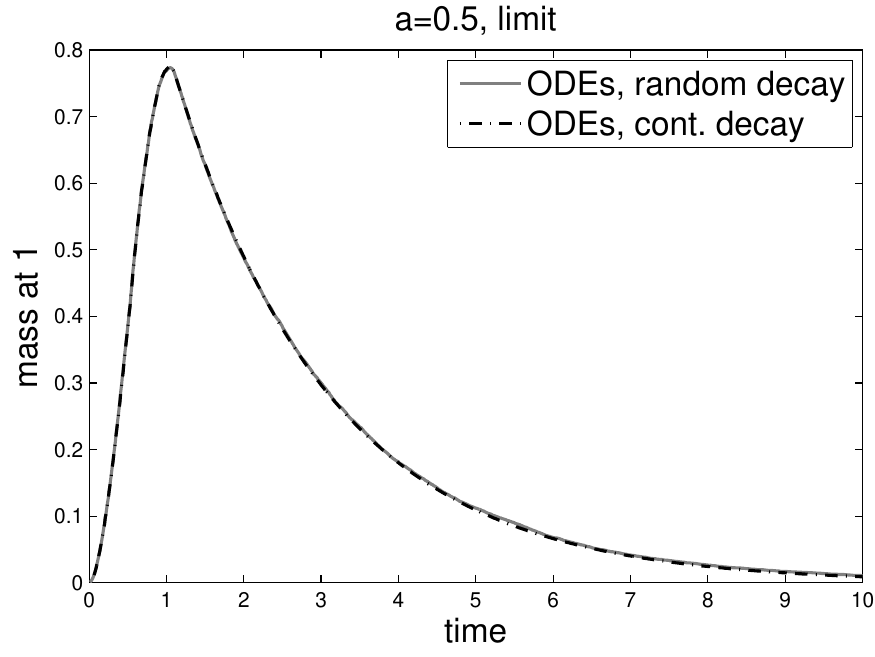}
        \caption{Mass at $x=1$ for the system with discrete initial measure in the limit of vanishing boundary layer. Comparison between taking away all mass of a particle at a random time after arrival, and continuous decay at rate $a$.}\label{fig:random limit at one}
\end{figure}
%%%%%%%%%%%%%%%%%%%%%%%%%

Differences between the solutions with random and continuous decay, respectively, can only occur at $x=1$. In the interior there is no absorption of mass, while we use the same initial data and the evolution of positions is dictated by the same velocity field. Therefore we only show the mass at $x=1$ as a function of time; see Figure \ref{fig:random limit at one}. In the graph hardly any difference is observed between the two systems.

\begin{appendix}

\section*{Appendices}

\section{Proof of Lemma \ref{lem:basic props indiv flow} -- Individualistic stopped flow}
\label{app:proofs lemmas}

\begin{proof} {\it (Lemma \ref{lem:basic props indiv flow})}\
\noindent ({\it i}):\ A quick approach to the stated result is the following: extend $v:[0,1]\to\R$ to $\bar{v}:\R\to\R$ by putting $\bar{v}(x):=v(0)$ if $x\leqs 0$ and $\bar{v}(x):=v(1)$ if $x\geqs 1$. Then $\bar{v}$ is a bounded Lipschitz extension of $v$ such that $\|\bar{v}\|_\infty=\|v\|_\infty$ and $|\bar{v}|_L=|v|_L$. Let $\bar{x}(t;x_0)$ be the unique (global) solution to \eqref{eq:indiv flow} with $v$ replaced by $\bar{v}$ with initial condition $x_0\in\R$ and $\bar{\Phi}_t:\R\to\R$ the associated solution semigroup. That is, $\bar{\Phi}_t(x_0):=\bar{x}(t;x_0)$. A classical argument involving Gronwall's Lemma yields ({\it i}) for $\bar{\Phi}_t$ instead of $\Phi_t$. Now, for $x_0\in[0,1]$,
\[
\Phi_t(x_0) = \min\bigl( \max(\bar{\Phi}_t,0),1\bigr).
\]
Thus $|\Phi_t|_L\leqs|\bar{\Phi}_t|_L$, using e.g. \cite{Dudley1} Lemma 4.

\noindent ({\it ii}):\
Let $t,s\in I_{x_0}$. Without loss of generality, assume that $t>s$.
\begin{align}
|x(t)-x(s)| &= |\Int{0}{t}{v(x(\sigma))}{d\sigma} - \Int{0}{s}{v(x(\sigma))}{d\sigma}|
\leqs  \Int{s}{t}{|v(x(\sigma))|}{d\sigma} \nonumber\\
&\leqslant\ \|v\|_\infty\, (t-s).\label{eq:time Lipschitz estimate}
\end{align}
If both $t,s\in\Rp$ are not in $I_{x_0}$, then inequality \eqref{eq:time Lipschitz estimate} is trivially satisfied. Suppose now that $s\in I_{x_0}$, while $t$ is not. Then $t>\tau_\partial(x_0)$, $\tau_\partial(x_0)\in I_{x_0}$ and
\[
|x(t)-x(s)| = |x(\tau_\partial(x_0))-x(s)| \leqs \|v\|_\infty\, (\tau_\partial(x_0)-s),
\]
according to \eqref{eq:time Lipschitz estimate}. Clearly $\tau_\partial(x_0)-s\leqs t-s$. The estimates are independent of $x_0\in[0,1]$. Thus we obtain \eqref{eq:uniform Lipschitz in time}.
\end{proof}

\section{Proof of central Lemma \ref{lem:crucial BL-function} -- Averaging over orbits}\label{regularisation-orbits}

Let $g:[0,1]\to\R$ be a piecewise bounded Lipschitz function, with finite set of discontinuities $S_g$. Number the points of $S_g\cup\{0,1\}$ in increasing order:
\[
x_0:=0<x_1<x_2<\dots<x_{n-1}<1=:x_n.
\]
For ease of exposition of the technical arguments, if $n=1$ we include an `artificial' point of discontinuity for $g$ in $(0,1)$ at position $x$ where $v(x)\neq 0$, such that we can assume $n\geqs 2$. The resulting partitioning of $[0,1]$ has mesh size $m_g:=\min_{1\leqs i\leqs n} (x_i-x_{i-1})$. The restriction of $g$ to the subinterval $(x_{i-1},x_i)$ has a unique Lipschitz extension to $[x_{i-1},x_i]$ that we denote by $g_i$. Note that in general $g_i(x_i)\neq g(x_i)\neq g_{i+1}(x_i)$ may hold. We assume that $v$ is a bounded Lipschitz velocity field on $[0,1]$ such that $v(x)\neq 0$ for $x\in S_g$. Let $(\Phi_s)$ be the individualistic stopped flow associated to $v$ (see Section \ref{sec:inidividualistic flow}). 

The establishment of the Lipschitz property of $g^\Phi_t$ turns out to be less straightforward as it might seem at first sight, resulting in the technical proof that is presented in this section. One might think of starting from the observation that
\begin{equation}\label{eq:naive approach}
g^\Phi_t(x) = \int_0^t g(\Phi_s(x))\,ds = \int_0^t \frac{g(\Phi_s(x))}{v(\Phi_s(x))}\,\frac{d\Phi_s(x)}{ds}\,ds 
= \int_x^{\Phi_t(x)} \frac{g(x')}{v(x')} \,dx'.
\end{equation}
However, \eqref{eq:naive approach} holds for only for $x$ that are not a sticking boundary point nor an interior steady state, and holds for $t\leqs \tau_\partial(x)$ only. Moreover, when there is an internal steady state, then \eqref{eq:naive approach} does not give an easy Lipschitz estimate for $|g^\Phi_t(x)-g^\Phi_t(y)|$ when $x$ and $y$ lie on differnt sides of this steady state. This strongly limits the applicability of \eqref{eq:naive approach}.\\

In our approach consider the {\em product flow} $(\Phi^\times_s)$ on $[0,1]\times [0,1]$ given by
\[
\Phi^\times_s(x,y) := (\Phi_s(x),\Phi_s(y))\qquad\mbox{for}\ s\geq 0
\]
and define for $1\leqs i,j\leqs n$ and $t\geqs 0$:
\begin{align*}
\overline{B}_{i,j} &:= \bigl\{ (x,y)\in[0,1]\times[0,1]\;|\; x_{i-1}\leqs x\leqs x_i,\ x_{j-1}\leqs y\leqs x_j\},\\
I^t_{i,j}(x,y) &:= \bigl\{ s\in[0,t]\;|\; \Phi^\times_s(x,y)\in\overline{B}_{i,j}\bigr\},\\
I^{t,\partial}_{i,j}(x,y) &:= \bigl\{ s\in[0,t]\;|\; \Phi^\times_s(x,y)\in\partial\overline{B}_{i,j}\setminus\Delta_{[0,1]}\bigr\}.
\end{align*}
Here, $\partial\overline{B}_{i,j}$ denotes the boundary of $\overline{B}_{i,j}$ and $\Delta_{[0,1]}:=\{(x,y)\in [0,1]\times [0,1]\;|\; x=y\}$ is the diagonal in $[0,1]\times[0,1]$. Then one has for any $x,y\in[0,1]$,

%Note that each $I^t_{i,j}(x,y)$ is either empty or a closed interval (possibly a singleton). Moreover, the pairwise intersection $I^t_{i,j}(x,y)\cap I^t_{i',j'}(x,y)$ is either empty or a singleton when $(i,j)\neq (i',j')$ because of the assumption $v(x_i)\neq 0$ for $x_i\in S_g$. This assumption also implies that $\{s\in [0,t]\;|\; \Phi^\times_s(x,y)\in \partial \overline{B}_{i,j}\}$ consists of at most two points, unless $i=j=1$ or $i=j=n$.

%\begin{align}
%\bigl| g^\Phi_t(x) - g^\Phi_t(y) \bigr| & = \left| \int_0^t g(\Phi_s(x)) - g(\Phi_s(y))\,ds\right| \nonumber\\
%& \leqs \sum_{1\leqs i,j\leqs n}\ \int_{I^t_{i,j}(x,y)} \bigl| g_i(\Phi_s(x)) - g_j(\Phi_s(y)) \bigr| \,ds \label{eq:crucial Lemma - sum of parts}.
%\end{align}
\begin{equation}
\bigl| g^\Phi_t(x) - g^\Phi_t(y) \bigr|
\leqs \sum_{1\leqs i,j\leqs n}\ \int_{I^t_{i,j}(x,y)} \bigl| g(\Phi_s(x)) - g(\Phi_s(y)) \bigr| \,ds \label{eq:crucial Lemma - sum of parts}.
\end{equation}
%\begin{equation}
%\bigl| g^\Phi_t(x) - g^\Phi_t(y) \bigr| = \left| \int_0^t g(\Phi_s(x)) - g(\Phi_s(y))\,ds\right| \nonumber\\
%\leqs \sum_{1\leqs i,j\leqs n}\ \int_{I^t_{i,j}(x,y)} \bigl| g(\Phi_s(x)) - g(\Phi_s(y)) \bigr| \,ds \label{eq:crucial Lemma - sum of parts}.
%\end{equation}
The terms in \eqref{eq:crucial Lemma - sum of parts} with $|i-j|\geqs 2$, $|i-j|=1$, and $i=j$ shall be estimated separately, since each requires a different approach. %Note that the first two cases are relevant only when $n\geqs2$, i.e. when there are discontinuities of $g$ in the interior of $[0,1]$.
The last two cases are the most delicate. The estimates are as follows:
\vskip 0.1cm
\noindent{\it Case $|i-j|\geqs 2$.}\\
For $s\in I^t_{i,j}(x,y)$, $\Phi_s(x)$ and $\Phi_s(y)$ are in intervals of the partitioning of $[0,1]$ that are not neighbours. So $|\Phi_s(x) - \Phi_s(y)|\geqs m_g$. This implies that in this case
\begin{align*}
\int_{I^t_{i,j}(x,y)} | g(\Phi_s(x))-g(\Phi_s(y))| \, ds & \leqs \frac{2}{m_g} \|g\|_\infty\,\int_{I^t_{i,j}(x,y)} |\Phi_s(x) - \Phi_s(y)|\, ds\\
& \leqs \frac{2}{m_g} \|g\|_\infty \, e^{|v|_Lt}\cdot |I^t_{i,j}(x,y)| \cdot |x-y|,
\end{align*}
according to Lemma \ref{lem:basic props indiv flow}. Here, $|A|$ denotes the Lebesgue measure of the (measurable) set $A$.
\vskip 0.1cm
\noindent{\it Case $|i-j|=1$.}\\
The following lemma provides the crucial observation for this case:
\begin{lemma}\label{lem:quasi-Lipschitz transit time}
Assume $n\geqs 2$. Then for all $t\geqs 0$ and $1\leqs i,j\leqs n$ such that $|i-j|=1$, there exist $L^t_{i,j}\geqs 0$ such that
\begin{equation}\label{eq:quasi-Lipschitz estimate}
|I^t_{i,j}(x,y)| \leqs L^t_{i,j}\, |x-y|
\end{equation}
for all $x,y\in[0,1]$. The functions $t\mapsto L^t_{i,j}$ are non-decreasing and locally bounded.
\end{lemma}

\noindent The terms in \eqref{eq:crucial Lemma - sum of parts} with $|i-j|=1$ can then be estimated as follows, where we limit our exposition to the case $j=i-1$:
\begin{equation}
\int_{I^t_{i,i-1}(x,y)} | g(\Phi_s(x))-g(\Phi_s(y))| \, ds \leqs 2\|g\|_\infty\cdot |I^t_{i,i-1}(x,y)|
\leqs\ 2\|g\|_\infty L^t_{i,i-1} \cdot |x-y|.
\end{equation}

\newpage
\noindent{\it Case $i=j$.}\\
Since $g$ is possibly discontinuous at the points $x_{i-1}$ and $x_i$ we have, using Lemma \ref{lem:basic props indiv flow},
\begin{align}
\nonumber \int_{I^t_{i,i}(x,y)} \bigl| g(\Phi_s(x)) - g(\Phi_s(y)) \bigr| \,ds \leqs& \int_{I^t_{i,i}(x,y)} | g_i(\Phi_s(x))-g_i(\Phi_s(y))| \, ds \\
\label{eq:est i=j bandry problem term} &\qquad + \int_{I^{t,\partial}_{i,i}(x,y)} | g(\Phi_s(x))-g(\Phi_s(y))| \, ds\\
\leqs & |g_i|_L e^{|v|_Lt}\cdot |I^t_{i,i}(x,y)|\cdot |x-y| + 2\|g\|_\infty\,|I^{t,\partial}_{i,i}(x,y)|.\label{eqn: est i=j bdry}
\end{align}

\noindent Note that $I^{t,\partial}_{i,i}(x,y)$ does not contain times $s\in[0,t]$ at which $\Phi^\times_s(x,y)\in\Delta_{[0,1]}$. For these times the integrand is zero, however. The following lemma provides the Lipschitz estimate for the second term in \eqref{eqn: est i=j bdry}, similar to Lemma \ref{lem:quasi-Lipschitz transit time}:

\begin{lemma}\label{lem:quasi-Lipschitz time bdry i=j new}
If $2\leqs i \leqs n-1$, then $|I^{t,\partial}_{i,i}(x,y)|=0$ for all $x,y\in[0,1]$. If $i=1$ and $v(0)\neq 0$, or $i=n$ and $v(1)\neq 0$, then for all $t\geqs 0$ there exist $L^t_i\geqs 0$ such that for all $x,y\in [0,1]$
\begin{equation}\label{eq:quasi-Lipschitz estimate i=j new}
|I^{t,\partial}_{i,i}(x,y)| \leqs L^t_i\, |x-y|.
\end{equation}
The function $t\mapsto L^t_i$ is non-decreasing and locally bounded.
\end{lemma}

\noindent The estimates for the various cases can now be put together when $v(x)\neq 0$ for $x\in\{0,1\}$, yielding
\begin{align*}
\bigl| g^\Phi_t(x) - g^\Phi_t(y) \bigr| \leqs&\ \frac{2}{m_g} \|g\|_\infty \cdot e^{|v|_Lt} \left(\sum_{|i-j|\geqs 2}  |I^t_{i,j}(x,y)|\right) \cdot |x-y|\\
&+ 2\|g\|_\infty \left(\sum_{|i-j|=1} L^t_{i,j}\right)\cdot |x-y|\\
&+ \max_{1\leqs i\leqs n} |g_i|_L \cdot e^{|v|_L t} \left(\sum_{i=1}^{n} |I^t_{i,i}(x,y)|\right)\cdot |x-y|  + 2\|g\|_\infty (L^t_1+L^t_n)\cdot|x-y|\\
\leqs&\ |x-y| \cdot G \cdot \left(\frac{2n^2}{m_g}\cdot t e^{|v|_L t} + 2\sum_{|i-j|=1} L^t_{i,j} + 2(L^t_1+L^t_n) \right),
\end{align*}
with $G:=\max(\|g\|_\infty, |g_1|_L, \ldots, |g_n|_L)$. Here we used that $\sum_{i,j} |I^t_{i,j}(x,y)|\leqs n^2t$. The functions $t\mapsto L^t_{i,j}$ are non-decreasing and locally bounded when $|i-j|=1$, according to Lemma \ref{lem:quasi-Lipschitz transit time}, while $t\mapsto L^t_1$ and $t\mapsto L^t_n$ are non-decreasing and locally bounded, due to Lemma \ref{lem:quasi-Lipschitz time bdry i=j new}. So for every $t\geq 0$, $\sup_{0\leqs s\leqs t} |g^\Phi_s|_L<\infty$. 

If $v(0)=0$ or $v(1)=0$, then $g$ is continuous at the boundary points where $v$ vanishes, by the assumptions imposed on $v$. Consider the case $v(0)=0$. Then the domain of integration in the integral in \eqref{eq:est i=j bandry problem term} for $i=1$ can be replaced by
\[
\tilde{I}^{t,\partial}_{1,1}(x,y) := I^{t,\partial}_{1,1}(x,y)\setminus\{(x,y)\in[0,1]\times[0,1]\;|\; x=0\ \mbox{or}\ y=0\},
\]
because of the continuity of $g$ at 0. Since $v(x_1)\neq 0$, we obtain $|\tilde{I}^{t,\partial}_{1,1}(x,y)|=0$ and the last term in \eqref{eqn: est i=j bdry} can be omitted in this case. A similar argument applies to the case $v(1)=0$ and $i=n$. Again, for every $t\geq 0$, $\sup_{0\leqs s\leqs t} |g^\Phi_s|_L<\infty$. The proof of Lemma \ref{lem:crucial BL-function} is now complete.
\vskip 0.3cm

\noindent {\bf Proofs of Lemma \ref{lem:quasi-Lipschitz transit time} and Lemma \ref{lem:quasi-Lipschitz time bdry i=j new}.}\\
First define for any $x,y\in[0,1]$:
\begin{equation}
\tau_y(x) := \inf\{s\geqs0\;:\; y=\Phi_s(x) \},
\end{equation}
with the convention that $\inf \emptyset=\infty$. If it is finite, then $\tau_y(x)$ is the arrival time at $y$ of the solution starting at $x$. In that case,
\begin{equation}\label{eq:integral first arrival time}
\tau_y(x) = \int_0^{\tau_y(x)} \frac{1}{v(\Phi_s(x))}\cdot \frac{d\Phi_s(x)}{ds}\, ds =\int_x^y \frac{1}{v(x')}\,dx',
\end{equation}
because $v(\Phi_s(x))\neq 0$ for any $0\leqs s\leqs \tau_y(x)$, otherwise $y$ is not reachable (note that $v$ Lipschitz implies that a steady state cannot be reached in finite time). For $t\geqs0$ define the truncation
\begin{equation}\label{eqn: def truncation}
t\wedge \tau_{y}(x) := \min(t,\tau_{y}(x))
\end{equation}

\begin{lemma}\label{lemma:Lipschitz prop tau}
Let $n\geqs 2$ and $t>0$.
\begin{enumerate}
\item[(a)] If $1\leqs i\leqs n-1$, then $x\mapsto t\wedge\tau_{x_i}(x)$ is continuous on $[x_{i-1},x_{i+1}]$ except at $x_i$. Moreover,
\begin{enumerate}
\item[(i)] If $v(x_i)>0$, then $t\wedge\tau_{x_i}$ is Lipschitz continuous on $[x_{i-1}, x_i]$.
\item[(ii)] If $v(x_i)<0$, then $t\wedge\tau_{x_i}$ is Lipschitz continuous on $[x_i,x_{i+1}]$.
\end{enumerate}
In either of the cases, $t\wedge\tau_{x_i}(x)=t$ on the remaining part of $[x_{i-1},x_{i+1}]$.
\item[(b)] If $v(0)<0$, then $x\mapsto t\wedge\tau_{x_0}(x)$ is Lipschitz continuous on $[x_0,x_1]$.
\item[(c)] If $v(1)>0$, then $x\mapsto t\wedge\tau_{x_n}(x)$ is Lipschitz continuous on $[x_{n-1},x_n]$.
\end{enumerate}
In all cases the corresponding Lipschitz constant of $t\wedge\tau_{x_i}$ on the stated interval is a non-decreasing function of $t$.
\end{lemma}

\begin{proof} {\it (a)}\ First assume that $v(x_i)>0$. Then $\tau_{x_i}$ is finite on the connected component of the set $S_i:=\{x\in[x_{i-i},x_i]\;|\; v(x)>0\}$ that contains $x_i$. According to \eqref{eq:integral first arrival time}, $x\mapsto\tau_{x_i}(x)$ is strictly decreasing and continuously differentiable function on $S_i\setminus\{x_i\}$. Hence, $t\wedge \tau_{x_{i}}$ is a Lipschitz function on $[x_{i-i},x_{i}]$, with Lipschitz constant on this interval given by
\begin{equation}
|t\wedge \tau_{x_{i}}|^{(i)}_L = \sup\left\{ \frac{1}{v(x')}\;\bigl|\; x'\in[x_{i-1},x_{i}],\ \tau_{x_{i}}(x')\leqs t\right\},
\end{equation}
because of \eqref{eq:integral first arrival time}. So $t\mapsto |t\wedge \tau_{x_{i}}|^{(i)}_L$ is a non-decreasing (and locally bounded) function. Note that $t\wedge \tau_{x_{i}}(x_{i})=0$, while $t\wedge \tau_{x_{i}}(x)=t$ for $x\in(x_{i},x_{i+1}]$.

A similar argument applies to the case $v(x_i)<0$. Now $x\mapsto \tau_{x_i}(x)$ is strictly increasing on the connected component of $S'_i:=\{x\in[x_i,x_{i+1}]\;|\; v(x)<0\}$ that contains $x_i$. Further details are left to the reader.

\noindent{\it (b) and (c):}\ The arguments are similar to the Lipschitz part of (a).
\end{proof}

\noindent {\bf Remark:}\ If $v(x_i)$ were 0, then $t\wedge\tau_{x_i}$ is constant $t$ on $[x_{i-1},x_{i+1}]\setminus\{x_i\}$, but discontinuous for $t>0$ at $x=x_i$, since $t\wedge\tau_{x_i}(x_i)=0$. Hence, it is neither Lipschitz on $[x_{i-1},x_i]$, nor on $[x_{i},x_{i+1}]$.

\begin{proof}{\it (Lemma \ref{lem:quasi-Lipschitz transit time}).}\ Because $I^t_{i,j}(x,y)=I^t_{j,i}(y,x)$, it suffices to consider the case $j=i-1$, $i\geqs 2$. Fix $x,y\in[0,1]$. We can assume that $I^t_{i,j}(x,y)\neq\emptyset$. First suppose that $(x,y)\not\in \overline{B}_{i,i-1}$. Then $t_0:=\inf(I^t_{i,j}(x,y))$ is the time of arrival of the product flow at $\partial\overline{B}_{i,i-1}$ before time $t$, when the flow starts at $(x,y)$. Consequently,
\begin{equation}\label{eq:reduction to boundary point}
|I^t_{i,i-1}(x,y)|  = |I^{t-t_0}_{i,i-1}(\Phi^\times_{t_0}(x,y))|.
\end{equation}
Since
\begin{equation}\label{eq:Lipschitz flow}
|\Phi_{t_0}(x)-\Phi_{t_0}(y)| \leqs e^{|v|_Lt_0}\cdot |x-y| \leqs e^{|v|_Lt}\cdot |x-y|,
\end{equation}
(cf. Lemma \ref{lem:basic props indiv flow}) it suffices to prove \eqref{eq:quasi-Lipschitz estimate} for $(x,y)\in\overline{B}_{i,i-1}$.\\

\noindent We now estimate $|I^t_{i,i-1}(x,y)|$ for $(x,y)\in\overline{B}_{i,i-1}$ under the assumption that $v(x_{i-1})>0$, by distinguishing three cases:\\

\noindent{\it Case $x_{i-1}\leqs x<x_i$ and $x_{i-2}<y<x_{i-1}$.}\\
Using the continuity of $v$ and $v(x_{i-1})>0$ and taking into account that the individualistic stopped flow will stay at the boundary points $x_0=0$ and $x_n=1$ of $[0,1]$ once the solution arrives at these points (when $v(0)<0$ and $v(1)>0$), one gets by careful consideration that
\begin{equation}\label{eqn: interval x=xi-1}
|I^t_{i,i-1}(x,y)| = \left\{
  \begin{array}{ll}
    \min\bigl(t\wedge \tau_{x_{i}}(x), t\wedge \tau_{x_{i-1}}(y)\bigr), & \hbox{$i=2$, $n>2$;} \\
    \min\bigl(t\wedge \tau_{x_{i-2}}(y), t\wedge \tau_{x_{i-1}}(y)\bigr), & \hbox{$i=n$, $n>2$;} \\
    t\wedge \tau_{x_{i-1}}(y), & \hbox{$i=n=2$;} \\
    \min\bigl(t\wedge \tau_{x_{i}}(x), t\wedge \tau_{x_{i-2}}(y), t\wedge \tau_{x_{i-1}}(y)\bigr), & \hbox{otherwise.}
  \end{array}
\right.
\end{equation}
Here we used that $v(x_{i-1})>0$. Hence,
\begin{align}
|I^t_{i,i-1}(x,y)| &\leqslant t\wedge \tau_{x_{i-1}}(y) = t\wedge \tau_{x_{i-1}}(y) - \underbrace{t\wedge \tau_{x_{i-1}}(x_{i-1})}_{=0} \label{eq:Lipschitz-like estimate 1}\\
& \leqslant |t\wedge \tau_{x_{i-1}}|^{(i-1)}_L \, |x_{i-1}-y| \label{eq:Lipschitz-like estimate 2}\\
& \leqs |t\wedge \tau_{x_{i-1}}|^{(i-1)}_L \, |x-y|.\label{eq:Lipschitz-like estimate 3}
\end{align}
Note that in \eqref{eq:Lipschitz-like estimate 1} -- \eqref{eq:Lipschitz-like estimate 3} the assumption $j=i-1$ is essential. To get \eqref{eq:Lipschitz-like estimate 2} we applied Lemma \ref{lemma:Lipschitz prop tau}(i).\\

%\noindent{\it Case $x_{i-2}<y<x_{i-1}$ and $x_{i-1}<x<x_i$.}\\
%In this case $|I^t_{i,i-1}(x,y)|$ is identical to \eqref{eqn: interval x=xi-1}. The subsequent estimate is slightly different:
%\begin{equation}
%|I^t_{i,i-1}(x,y)| \leqslant t\wedge \bar{\tau}_{x_{i-1}}(y) = t\wedge \bar{\tau}_{x_{i-1}}(y) - \underbrace{t\wedge \bar{\tau}_{x_{i-1}}(x_{i-1})}_{=0} \leqslant |t\wedge \bar{\tau}_{x_{i-1}}|^{i-1}_L \, \underbrace{|x_{i-1}-y|}_{\leqslant|x-y|}.
%\end{equation}

\noindent{\it Case $x_{i-1}\leqs x\leqs x_i$ and $y=x_{i-1}$.}\\
In this case $|I^t_{i,i-1}(x,y)|=0$ due to the sign of $v$ at $x_{i-1}$.\\
\\
\noindent{\it Case $x=x_i$ or $y=x_{i-2}$ (or both).}\\
Note that $|x-y|\geqslant m_g$. Hence,
\begin{equation}
|I^t_{i,i-1}(x,y)| = \frac{ |I^t_{i,i-1}(x,y)|}{|x-y|}\cdot |x-y| \leqslant \dfrac{t}{m_g}\,|x-y|.
\end{equation}

\noindent From the above cases we deduce that if $v(x_{i-1})>0$, then
\begin{equation}\label{eqn: est interval v>0}
|I^t_{i,i-1}(x,y)| \leqslant \max(\dfrac{t}{m_g}, |t\wedge \tau_{x_{i-1}}|^{(i-1)}_L)\,\cdot |x-y|,
\end{equation}
holds for all $(x,y)\in\overline{B}_{i,i-1}$. The pre-factor on the right-hand side in \eqref{eqn: est interval v>0} is non-decreasing and locally bounded in $t$.\\
\\
If $v(x_{i-1})<0$, then $t\wedge \tau_{x_{i-1}}$ is not Lipschitz on $[x_{i-2},x_{i-1}]$, but on $[x_{i-1},x_i]$ instead (Lemma \ref{lemma:Lipschitz prop tau}). We denote its Lipschitz constant on the latter interval by $|t\wedge \tau_{x_{i-1}}|^{(i)}_L$, which is a non-decreasing function of $t$. The estimates for $|I^t_{i,i-1}(x,y)|$ follow, \textit{mutatis mutandis}, from distinguishing between similar cases as for $v(x_{i-1})>0$. We obtain
\begin{equation}\label{eqn: est interval v<0}
|I^t_{i,i-1}(x,y)| \leqslant \max(\dfrac{t}{m_g}, |t\wedge \tau_{x_{i-1}}|^{(i)}_L)\,|x-y|,
\end{equation}
for all $(x,y)\in\overline{B}_{i,i-1}$, where $\max(\dfrac{t}{m_g}, |t\wedge \tau_{x_{i-1}}|^{(i)}_L)$ is non-decreasing and locally bounded in $t$.\\
\\
Since $v(x_{i-1})\neq 0$ by assumption on the velocity field, the combination of \eqref{eqn: est interval v>0} and \eqref{eqn: est interval v<0} yields the result stated in the lemma.
\end{proof}

\begin{proof}{\it (Lemma \ref{lem:quasi-Lipschitz time bdry i=j new}).}
Fix $x,y\in[0,1]$ such that $x\neq y$. Since $I^{t,\partial}_{i,i}(x,y)=I^{t,\partial}_{i,i}(y,x)$ we can assume $x<y$. Moreover, we can assume $I^{t,\partial}_{i,i}(x,y)\neq\emptyset$. Furthermore, it suffices to prove Lemma \ref{lem:quasi-Lipschitz time bdry i=j new} for $(x,y)\in\overline{B}_{i,i}$ due to similar arguments as in \eqref{eq:reduction to boundary point}--\eqref{eq:Lipschitz flow}.

If $n\geqs 3$ and $2\leqs i\leqs n-1$, then $v(x')\neq 0$ and $v(y')\neq 0$ for any $(x',y')\in\partial\overline{B}_{i,i}$ by assumption. Consequently, $|I^{t,\partial}_{i,i}(x,y)|=0$ in this case.

Consider now $i=1$. Recall that we have $n\geqs 2$, $x_1\in(0,1)$ and $v(x_1)\neq 0$. If $v(0)>0$, i.e. $0$ is non-sticking, then $|I^{t,\partial}_{1,1}(x,y)|=0$ for all $x,y\in[0,1]$ with $x<y$. If $v(0)<0$, then for $x,y\in[0,1]$ with $x<y$,
\begin{equation}\label{eq:cases boundedary 0}
I^{t,\partial}_{1,1}(x,y) = \begin{cases}
\ \{t\wedge\tau_{x_1}(y)\}, & \mbox{if}\ \tau_{x_1}(y)\leqs t\wedge \tau_{x_0}(x),\\
\ \bigl[t\wedge \tau_{x_0}(x),\min(t\wedge\tau_{x_0}(y),t\wedge\tau_{x_1}(y))\bigr], & \mbox{if}\ \tau_{x_0}(x)\leqs t\wedge\tau_{x_1}(y),\\
\ \emptyset, & \mbox{otherwise}.
\end{cases}
\end{equation}
Here we used that the flow does not stop at $x_1\neq 1$. Moreover, in the second case in \eqref{eq:cases boundedary 0} we use that $\Delta_{[0,1]}$ is excluded in the definition of $I^{t,\partial}_{1,1}(x,y)$. Thus
\begin{equation}
|I^{t,\partial}_{1,1}(x,y)| = \begin{cases}
\ \min(t\wedge\tau_{x_0}(y),t\wedge\tau_{x_1}(y))-t\wedge \tau_{x_0}(x), & \mbox{if}\ \tau_{x_0}(x)\leqs t\wedge\tau_{x_1}(y),\\
\ 0, & \mbox{otherwise}.
\end{cases}
\end{equation}
Hence, if $v(0)<0$, then
\begin{equation}
|I^{t,\partial}_{1,1}(x,y)| \leqslant |t\wedge \tau_{x_0}(y) - t\wedge \tau_{x_0}(x)| \leqslant |t\wedge \tau_{x_0}|_L^{(1)}\cdot |x-y|.
\end{equation}
As argued in the proof of Lemma \ref{lem:quasi-Lipschitz transit time}, $t\mapsto|t\wedge \tau_{x_0}|_L^{(1)}$ is non-decreasing, locally bounded.\\
\\
Consider $i=n$. If $v(1)<0$, then $|I^{t,\partial}_{n,n}(x,y)|=0$ for all $x,y\in[0,1]$ with $x<y$. If $v(1)>0$, then
\begin{equation}\label{eq:cases boundedary 1}
I^{t,\partial}_{n,n}(x,y) = \begin{cases}
\ \{t\wedge\tau_{x_{n-1}}(x)\}, & \mbox{if}\ \tau_{x_{n-1}}(x)\leqs t\wedge \tau_{x_n}(y),\\
\ \bigl[t\wedge \tau_{x_n}(y),\min(t\wedge\tau_{x_{n-1}}(x),t\wedge\tau_{x_n}(x))\bigr], & \mbox{if}\ \tau_{x_n}(y)\leqs t\wedge\tau_{x_{n-1}}(x),\\
\ \emptyset, & \mbox{otherwise}.
\end{cases}
\end{equation}
Consequently, with similar argument as above, if $v(1)>0$, then
\begin{equation}
|I^{t,\partial}_{n,n}(x,y)| \leqslant |t\wedge \tau_{x_n}(y) - t\wedge \tau_{x_n}(x)| \leqslant |t\wedge \tau_{x_n}|_L^{(n)}\cdot |x-y|.
\end{equation}
As argued before, each Lipschitz constant is non-decreasing and locally bounded in time.
This completes the proof.
\end{proof}

\section{Integration of measure-valued maps}
\label{sec:integration meas-valued maps}

Let $(X,\Sigma)$ be a measurable space and $S$ a Polish space. We refer to \cite{Diestel-Uhl} for the basic results on Bochner integration. The following result shows that $\|\cdot\|_\BL^*$ is a good norm from the point of view of integration.

\begin{proposition}\label{prop:measurability}
For any map $p:X\to\CM(S)$ the following statements are equivalent:
\begin{enumerate}
\item[({\it i})] $p$ is Bochner measurable as map into $\CMc(S)_\BL$;
\item[({\it ii})] For each bounded measurable function $\varphi$ on $S$, the map $x\mapsto \pair{p(x)}{\varphi}$ is measurable;
\item[({\it iii})] For each Borel measureable $E\subset S$, $x\mapsto p(x)(E)$ is measurable.
\end{enumerate}
\end{proposition}
\begin{proof}
A detailed proof is given in \cite{Worm-thesis:2010}. A version for positive measures is proven in \cite{Hille-Worm:2009}, Proposition 2.5.
\end{proof}

If $\mu$ is a $\sigma$-finite positive measure on $(X,\Sigma)$, $x\mapsto p(x)$ is Bochner measurable and $x\mapsto \|p(x)\|_\BL^*$ is integrable with respect to $\mu$, then $p$ is Bochner integrable and
\begin{equation}
\bigl\| \int_X p(x)\,d\mu(x)\bigr\|_\BL^*\ \leqs \int_X \|p(x)\|_\BL^*\,d\mu(x)
\end{equation}
(see e.g. \cite{Diestel-Uhl}). Because $S$ is separable, $\CMc(S)_\BL$ is separable. Therefore there exists a countable subset $\mathcal{N}\subset \CMc(S)_\BL^*$ that is norming:
\[
\|\varphi\|_\BL^* = \sup \bigl\{ |\pair{\varphi}{f}|\;:\; f\in\mathcal{N} \bigr\}
\]
for all $\varphi\in \CMc(S)_\BL$. Since $\CMc(S)_\BL^*\simeq \BL(S)$ (\cite{Hille-Worm:2009}, Theorem 3.7, p.360), we may consider $\mathcal{N}$ as subset of $\BL(S)$. In particular, if $p:X\to\CM(S)$ satisfies the conditions of Proposition \ref{prop:measurability}, then $x\mapsto \|p(x)\|_\BL^*$ is measurable.

\begin{proposition}\label{prop:TV-estimates}
Let $\mu$ be a $\sigma$-finite positive measure on $(X,\Sigma)$ and let $p:X\to\CM(S)$ satisfy any of the equivalent conditions in Proposition \ref{prop:measurability}. Then $x\mapsto \|p(x)\|_\TV$ is Borel measurable. Moreover, if the latter function is in $L^1(X,\mu)$, then:
\begin{enumerate}
\item[({\it i})] For each Borel set $E\subset S$, the set-wise integral
$\nu(E) := \int_X p(x)(E)\,d\mu(x)$
is defined and yields a finite Borel measure $\nu$.
\item[({\it ii})] The map $x\mapsto p(x)$ is Bochner integrable and the Bochner integral $\nu':=\int_X p(x)\,d\mu(x)$ in $\CMc(S)_\BL$ equals $\nu$. In particular, for any Borel set $E$ in $S$,
\begin{equation}\label{eq:set within integral}
\left(\int_X p(x)\,d\mu(x)\right)(E) = \int_X p(x)(E) \,d\mu(x).
\end{equation}
\item[({\it iii})] \[
\bigl\| \; \int_X p(x)\,d\mu(x)\; \bigr\|_\TV\ \leqs \ \int_X \|p(x)\|_\TV\,d\mu(x).
\]
\end{enumerate}
\end{proposition}

\begin{proof}
Because $S$ is Polish, there exists a countable algebra $\mathcal{A}$ of Borel sets that generates the Borel $\sigma$-algebra $\mathcal{B}(S)$ (cf. \cite{Bogachev}, Example 6.5.2). Then for any Borel measure $\mu$, every $\eps>0$ and Borel set $E$ there exists $A\in\mathcal{A}$ such that $|\mu(E)-\mu(A)|<\eps$. Therefore (cf. \cite{Bogachev-I}, p.176),
\[
\|p(x)\|_\TV = \sup_{E\in\mathcal{B}(S)} p(x)(E) - \inf_{E\in\mathcal{B}(S)} p(x)(E) = \sup_{A\in \mathcal{A}} p(x)(A) - \inf_{A\in\mathcal{A}} p(x)(A).
\]
The functions $p^{\mathcal{A}}:x\mapsto \sup_{A\in \mathcal{A}} p(x)(A)$ and $p_{\mathcal{A}}:x\mapsto \inf_{A\in \mathcal{A}} p(x)(A)$ are measurable as pointwise supremum of a countable collection of measurable functions (see Proposition \ref{prop:measurability}). So $x\mapsto \|p(x)\|_\TV$ is measurable.

\noindent{\it (i):}\ The integral defining $\nu$ converges, because $|p(x)(E)|\leqs \|p(x)\|_\TV$. $\sigma$-Additivity of $\nu$ is obtained through Lebesgue's Dominated Convergence Theorem.

\noindent{\it (ii):}\ $x\mapsto \|p(x)\|_\BL^*$ is measureable and dominated by the $\mu$-integrable function $\|p(x)\|_\TV$, so $p$ is indeed Bochner integrable. For any $f\in\BL(S)$,
\[
\pair{\nu'}{f} = \int_X \pair{p(x)}{f}\, d\mu(x) = \pair{\nu}{f}.
\]
The first step holds because $f$ defines a continuous functional on $\CMc(S)_\BL$, the second because $f$ is the pointwise limit of a sequence of step functions. Two finite Borel measures that coincide on $\BL(S)$ are identical (e.g. \cite{Dudley1}, Lemma 6). For $\nu$, clearly
\[
\int_S f\,d\nu = \int_X \pair{p(x)}{f}\,d\mu(x)
\]
for any bounded measurable function $f$. Equation \eqref{eq:set within integral} follows by taking $f=\ind_E$.

\noindent{\it (iii):}\ From part ({\it ii}) it follows that for any $A\in\mathcal{A}$ (introduced above),
\[
\int_X p_{\mathcal{A}}(x)\,d\mu(x)\ \leqs\ \nu(A)\ \leqs\ \int_X p^{\mathcal{A}}(x)\,d\mu(x).
\]
Therefore,
\[
\|\nu\|_\TV  = \sup_{A\in\mathcal{A}} \nu(A) - \inf_{A\in\mathcal{A}} \nu(A) \leqs \int_X p^{\mathcal{A}}(x) - p_{\mathcal{A}}(x)\, d\mu(x)
 = \int_X \|p(x)\|_\TV\, d\mu(x).
\]
\end{proof}

Moreover, for any continuous map $P:\CM^+(S)_\BL\to\CM^+(S)_\BL$ that is additive and positively homogeneous, i.e. $P(a \mu)= a P(\mu)$ for $a\geqs 0$, one has
\begin{gather}
P\bigl(\int_X p(x)\,d\mu(x)\bigr) = \int_X P[p(x)]\,d\mu(x).
\end{gather}

\end{appendix}

\end{document}